\documentclass[10pt]{article}
\usepackage{amsfonts}
\usepackage{amsmath,hhline,epsfig,latexsym}
\usepackage{amssymb}
\usepackage{graphicx}
\usepackage[english]{babel}
\usepackage{hyperref}
\usepackage{xcolor}
\usepackage{ntheorem}
\makeatletter
\renewcommand{\paragraph}{\@startsection{paragraph}{4}{0ex}%
   {-3.25ex plus -1ex minus -0.2ex}%
   {1.5ex plus 0.2ex}%
   {\normalfont\normalsize\bfseries}}
\makeatother

\stepcounter{secnumdepth}
\stepcounter{tocdepth}
\usepackage{amsmath}
\providecommand{\U}[1]{\protect\rule{.1in}{.1in}}
\newtheorem{thm}{Theorem}
\newtheorem{prop}{Proposition}

\newtheorem{lem}{Lemma}

\newtheorem{claim}{Claim}
\newenvironment{proof}[1][Proof]
{\noindent\textbf{#1:} }{\hfill\rule{0.5em}{0.5em}}

\newtheorem{theorem}{Theorem}[section]
\newtheorem{remark}[theorem]{Remark}

\begin{document}


\title{\bf On the Uniform Controllability of the Inviscid and Viscous Burgers-$\alpha$ Systems} 


\author{
\textsc{Raul K.C. Ara\'ujo}\thanks{Department of Mathematics, Federal University of Pernambuco, UFPE, CEP 50740-545, Recife,
PE, Brazil. E-mail: {\tt  raul@dmat.ufpe.br}. Partially supported by CNPq (Brazil).}
\quad
\and
	\textsc{Enrique Fern\' andez-Cara}\thanks{University of Sevilla, Dpto. E.D.A.N, Aptdo 1160, 41080 Sevilla, Spain. E-mail: {\tt cara@us.es}. Partially supported by grant MTM$2016$-$76990$-P, Ministry of Economy and Competitiveness (Spain).}
	\and
	\textsc{Diego A. Souza}\thanks{Department of Mathematics, Federal University of Pernambuco, UFPE, CEP 50740-545, Recife,
PE, Brazil. E-mail: {\tt diego.souza@dmat.ufpe.br}. D.A. Souza was partially supported by CNPq - Brazil by the grant $313148$/$2017$-1, by Propesq - UFPE edital Qualis A, CAPES-PRINT, \#$88881.311964/2018-01$. }
}

\maketitle


\begin{abstract}
    This work is devoted to prove the global uniform exact controllability of the inviscid and viscous Burgers-$\alpha$ systems. 
    The state $y$ is the solution to a regularized Burgers equation, where the transport velocity $z$ consists of a filtered version of $y$ -- specifically 
  	$z=(Id-\alpha^2\partial^2_ {xx})^{-1}y$ with $\alpha>0$ being a small parameter -- in place of $y$. 
  	First, a global uniform exact controllability result for the nonviscous Burgers-$\alpha$ system with three scalar controls is obtained, using the return method. 
	Then, global exact controllability to constant states of  the viscous system is deduced from a local exact controllability result and a global approximate controllability
	result for smooth initial and target states.
\end{abstract}

\


\tableofcontents


\section{Introduction}

	Let $L>0$ and $T>0$ be given. Let us present the notations used along this work. The symbols $C$, $\widehat C$ and $C_{i}$, $i = 0, 1, \dots$ stand for positive constants
	(usually depending on $L$ and~$T$). For any $r\in \left[ 1,+\infty\right]$ and any given Banach space $X$, $\|\cdot\|_{L^{r}(X)}$
	will denote the usual norm in Lebesgue-Bochner space $L^{r}(0,T;X)$. In particular, the norms in $L^{r}(0,L)$ and $L^{r}(0,T)$  will be denoted by $\|\cdot\|_{r}$.
	In this paper, we will consider the following two families of controlled systems:
\begin{equation}\label{es14}
          \left\{
	          \begin{array}{lcl}
	          	\noalign{\smallskip}\displaystyle
        			y_t + zy_x = p(t) 			&	\mbox{in}	& (0,T)\times(0,L),\\
			\noalign{\smallskip}\displaystyle
                        	z - \alpha^2 z_{xx} = y		& 	\mbox{in}	& (0,T)\times(0,L),\\ 
                        	\noalign{\smallskip}\displaystyle
			z(\cdot,0) = y(\cdot,0) = v_l	& 	\mbox{in}	& (0,T),\\
			\noalign{\smallskip}\displaystyle
			z(\cdot,L) = y(\cdot,L) = v_r	& 	\mbox{in}	&(0,T),\\
                		\noalign{\smallskip}\displaystyle
			y(0,\cdot) = y_0 			&	\mbox{in}	& (0,L)
          	\end{array}
         \right.
\end{equation}
         	and
\begin{equation}\label{viscous}
          \left\{
          	\begin{array}{lcl}
          	\noalign{\smallskip}\displaystyle
		y_t - \gamma y_{xx} + zy_x = p(t) 		&\mbox{in}	&(0,T)\times (0,L),\\
		\noalign{\smallskip}\displaystyle
		z - \alpha^2 z_{xx} = y 		&\mbox{in}	&(0,T)\times (0,L),\\
		\noalign{\smallskip}\displaystyle
          		z(\cdot,0) =  y(\cdot,0) = v_l 	&\mbox{in}	&(0,T),\\
		\noalign{\smallskip}\displaystyle
          		z(\cdot,L) = y(\cdot,L) = v_r	&\mbox{in}	&(0,T),\\
		\noalign{\smallskip}\displaystyle
          		y(0,\cdot) = y_0 			&\mbox{in}	&(0,L).
          	\end{array}
          \right.
\end{equation}
	These are the so called inviscid and viscous Burgers-$\alpha$ systems. The pairs $(y,z)$ and the triplets $(p,v_l,v_r)$ respectively stand for the corresponding 
	states and controls. 
	
	The physical motivations these systems are explained in \cite{bhat}, where it is shown that \eqref{es14} and \eqref{viscous} can be viewed as an asymptotically 
	equivalent approximation of the shallow water equations. Thus, $z$ can be viewed as the fluid velocity in the $x$ direction (or equivalently the height of the free surface of the 
	fluid above a flat bottom). The parameter $\gamma>0$ is the fluid viscosity and the role of $\alpha$ is to regularize the transport velocity $z$. Indeed, \eqref{es14} and \eqref{viscous} can be 
	regarded as nonlinear regularizations of the inviscid and viscous Burgers equation, see \cite{bhat}. Therefore, it is natural to try to deduce control properties and/or 
	estimates independent of $\alpha$. For simplicity, throughout this paper we will take $\gamma =1$ (all the results can be extended without difficulty to 
	the case where  $\gamma$ is an arbitrary positive number).

     	On the other hand, systems like \eqref{es14} and \eqref{viscous} can also be viewed as simplified 1D versions of the so called Leray-$\alpha$ system introduced some time ago to describe turbulent
     flows as an alternative to the classical averaged Reynolds models, see \cite{lerayalpha, Holm, Foias,manley}. It is also important to highlight that Leray-$\alpha$ models are related to the systems analyzed by Leray in \cite{Leray} to prove the existence of Navier-Stokes equations.


	Our two main results deal with the global uniform exact controllability (with respect to  $\alpha$) for systems \eqref{es14} and \eqref{viscous}. 
	More precisely, one has:
\begin{thm}\label{thm1}
	Let $\alpha > 0$ and $T>0$ be given. The inviscid Burgers-$\alpha$ system \eqref{es14} is globally exactly controllable in $C^1$. That is, 
	for any given $y_0, y_T \in C^1([0,L])$, there exist a time-dependent control $p^{\alpha} \in C^0([0,T])$, a couple of boundary controls $(v_l^{\alpha},v_r^{\alpha}) \in C^1([0,T];\mathbb{R}^2)$ and an associated 
	state $(y^{\alpha}, z^{\alpha}) \in C^1([0,T]\times [0,L];\mathbb{R}^2)$ satisfying \eqref{es14} and
       \begin{equation}\label{mg3}
       		y^{\alpha}(T,\cdot) = y_T\quad\mbox{in}\quad (0,L).
        \end{equation}
        	 Moreover, there exists a positive constant $C>0$ (independent of $\alpha$) such that
\begin{equation*}\label{v-unif}
    \|(z^\alpha,y^\alpha)\|_{C^1([0,T]\times[0,L];\mathbb{R}^2)}+	\|p^{\alpha}\|_{C^0([0,T])}+
	\|(v_l^{\alpha},v_r^{\alpha})\|_{C^1([0,T];\mathbb{R}^2)}\leq C.
\end{equation*}
\end{thm}
\begin{thm}\label{thm2}
	Let  $\alpha > 0$ and $T>0$ be given. The viscous Burgers-$\alpha$ system \eqref{viscous} is globally exactly controllable in $L^\infty$ to constant trajectories. That is, 
	for any given $y_0\in L^\infty(0,L)$ and $N\in \mathbb{R}$, there exist controls $p^{\alpha} \in C^0([0,T])$ and $(v_l^{\alpha}, v_r^{\alpha}) \in H^{3/4}(0,T; \mathbb{R}^2)$ 
	and associated states $(y^{\alpha},z^{\alpha}) \in L^2(0,T;H^1(0,L;\mathbb{R}^2))\cap L^\infty(0,T;L^\infty(0,L;\mathbb{R}^2))$ 
	satisfying \eqref{viscous}, 
	\begin{equation}\label{mg}
        		y^{\alpha}(T,\cdot)= z^{\alpha}(T,\cdot) = N \quad\mbox{in}\quad (0,L)
        \end{equation}
    and the following estimates
\begin{equation*}\label{v-unif}
    \|p^{\alpha}\|_{C^0([0,T])}+
	\|(v_l^{\alpha},v_r^{\alpha})\|_{H^{3/4}([0,T];\mathbb{R}^2)}\leq C,
\end{equation*}
    where $C$ is a positive constant independent of $\alpha$.
    Moreover, if $y_0\in H^1_0(0,L)$ then the same conclusion holds with the
    $(y^{\alpha},z^{\alpha})\in L^2(0,T;H^2(0,L;\mathbb{R}^2))\cap H^1(0,T;L^2(0,L;\mathbb{R}^2)).$
\end{thm}
	\begin{remark}
		\textrm{We will see in the proofs of the above results that the distributed control $p^\alpha$ is independent of $\alpha$, it only depends on $T$, $L$, the initial condition and the target state. }
	\end{remark}
		In this paper, we are going to deal with situations that lead to new difficulties compared to previous
	works on nonlinear parabolic equations. Let us discuss these differences:
\begin{itemize}
	\item Nonlocal nonlinearities. In the \eqref{es14} and \eqref{viscous}, the usual convective term is replaced and a filtered (averaged) velocity appears.
As a consequence, the arguments in \cite{Chapouly} must be modified, as shown below.
	
	\item Uniform controllability. Performing careful estimates of the controls, global uniform controllability results are obtained. This way, we are able to generalize some previous control results  to the context of nonlinear parabolic equations with nonlocal nonlinearities, see \cite{Chapouly,marbach}. 
\end{itemize}


	For completeness, let us mention some previous works on the control
	of our main systems and other similar models. There are a lot of important works dealing with the controllability properties of parabolic equations and systems,
	see \cite{doubova, fabre, zuazua, Fursikov} and the inviscid and viscous Burgers equation,
	see \cite{ Chapouly,xiang, Dias, Cara, Fursikov, Glass1, Guerrero, Horsin, marbach}. In the context of regularized Burgers equation, the local uniform null controllability for the viscous system \eqref{viscous} is studied in \cite{Araruna}. Finally, let us mention that, in \cite{diego},
	this local result was extended to any $b$-family (a collection of viscous Burgers-$\alpha$ equations equipped with additional nonlinear stretching term). 
	
	The rest of this paper is organized as follows.
	In~Section~\ref{sec:preliminares}, we prove some results concerning the existence, uniqueness and regularity of the solution to the viscous and inviscid
	Burgers-$\alpha$ systems.
	Sections~\ref{sec:inviscid} and \ref{sec:viscous} deal with the proofs of Theorems~\ref{thm1} and~\ref{thm2}, respectively.
	Finally, in~Section~\ref{sec:comments}, we present some additional comments and questions.


\section{Preliminaries} \label{sec:preliminares}


\subsection{Notations and Classical Results}

	Let us denote by $C^0_b(\mathbb{R})$ the Banach space of bounded continuous functions on $\mathbb{R}$ and let $C_x^{0,1}([0,T]\times\mathbb{R})$ be
	the space of functions $f : [0,T]\times \mathbb{R}\mapsto\mathbb{R}$ that are continuous in $x$ and $t$  and  globally Lipschitz-continuous in space, with Lipschitz constant independent of $t$.
	
	In the sequel, for any given $f\in C^0([0,T]\times\mathbb{R})$, the associated {\it flux function} $ \Phi = \Phi(s;t,x)$ is defined as follows:
\begin{equation}\label{flow3}
        \left\{
        		\begin{array}{l}
        			\noalign{\smallskip}\displaystyle	
     			\frac{\partial\Phi}{\partial t}(s;t,x) = f(t, \Phi(s;t,x)),\\
       			\noalign{\smallskip}\displaystyle	
        			\Phi^*(s;s,x) = x, 
        		\end{array}
        \right.
\end{equation}
         The mapping $\Phi$ contains all the information on the trajectories of the 
        particles transported by the velocity $f$. Furthermore, we have the following existence, uniqueness and regularity result:
        
\begin{prop}[Theorem $10.19$, \cite{ode}]\label{flow4}
        Assume that $f\in C_x^{0,1}([0,T]\times\mathbb{R})$ and $\frac{\partial f}{\partial x}$ belongs to \,\,\,\, $C^0([0,T]\times \mathbb{R})$. Then, there exists a unique flux associated to $f$, that is, a unique function $\Phi : [0,T]\times [0,T]\times \mathbb{R}\mapsto\mathbb{R}$ satisfying \eqref{flow3} for all $(s,x)\in [0,T]\times\mathbb{R}$. Moreover, $\Phi \in C^1([0,T]\times[0,T]\times\mathbb{R})$.
\end{prop}

    Under the assumptions of Proposition \ref{flow4}, it is well known that, for each $s,t$ in $[0,T]$, the mapping 
    $\Phi(s;t,\cdot) :\mathbb{R}\mapsto\mathbb{R}$ is a diffeomorphism, with
    \[
        \Phi(s;t,\cdot)^{-1} = \Phi(t;s,\cdot).
    \]
        
        Let us now recall a result from  Bardos and Frisch \cite{bardos}. To this purpose, let us first note that, for any given Banach space $X$ with 
        norm $\|\cdot\|_X$ and any function $u \in C^1([0,T]; X)$, the following inequality holds:
\begin{equation}\label{es25}
	\dfrac{d}{dt^+}\|u(t)\|_{X} \leq \left\|u_t(t)\right\|_{X}\quad\mbox{in}\quad(0,T),
\end{equation}
	where $d/dt^+$ represents the right derivative. 
 \begin{prop}[Lemma $1$, \cite{bardos}]\label{es28}
	 Let $v\in C^0([0,T]; C^{0,1}_b(\mathbb{R}))\cap C^{0,1}_x([0,T]\times\mathbb{R})$ and $g\in C^0([0,T]; C_b^0(\mathbb{R}))$ be given. Then, any solution
	 $y\in C^0([0,T];C_b^1(\mathbb{R}))\cap C^1([0,T];C_b^0(\mathbb{R}))$ to the equation
	   \begin{equation} \label{es26}
	   y_t + vy_x = g \quad\mbox{in}\,\, (0,T)\times\mathbb{R}
	   \end{equation}
	   satisfies the following inequality
	   \begin{equation*}
	   \dfrac{d}{dt^+}\|y(t,\cdot)\|_{C_b^0(\mathbb{R})} \leq \|g(t,\cdot)\|_{C_b^0(\mathbb{R})}\quad\mbox{in}\,\,(0,T). 
	   \end{equation*}
	 \end{prop}
	 \begin{proof}
	 Let $\Phi$ be the flow associated to $v$. For any $(s,t,x)\in [0,T]\times[0,T]\times\mathbb{R}$, we have by \eqref{es26} that
	 \begin{align*}
	 \dfrac{d}{dt}y(t,\Phi(s;t,x)) &= g(t,\Phi(s;t,x)). 
	 \end{align*}
	Using this identity and the fact that $\Phi(s;t,\cdot)$ is a diffeomorphism, we get
	 \begin{align*}
	\left\|\dfrac{d}{dt}y(t,\cdot)\right\|_{C_b^0(\mathbb{R})} &\leq \|g(t,\cdot)\|_{C_b^0(\mathbb{R})}.
	 \end{align*}
	 Now, the result follows easily from this and from \eqref{es25}.
	 \end{proof}
	
	The last result of this section is an immediate consequence of the Banach Fixed-Point Theorem:
\begin{thm}\label{fixed-point}
	Let $(E,\|\cdot\|_E)$ and $(F,\|\cdot\|_F)$ be Banach spaces with $F$ continuously embedded in $E$. Let $B$ be a subset of $F$ and let $G: B \mapsto B$ be a mapping such that
$$
	\|G(u)-G(v)\|_E\leq \gamma \|u-v\|_E\,\, ~\forall~u,v\in B,\hbox{ for some } \gamma \in [0,1).
$$
	Let us denote by $\widetilde B$ the closure of $B$ for the norm $\|\cdot\|_E$.
	Then, $G$ can be uniquely extended to a continuous mapping $\widetilde G: \widetilde B \mapsto \widetilde B$ that possesses a unique fixed-point in $\widetilde B$.
\end{thm}

 

 
 \subsection{Well-Posedness of the Viscous Burgers-$\alpha$ System}\label{subsec:wellposedness_viscous}

  	    Let us introduce the Hilbert space 
  	    $E := H^{3/4}(0,T)\times H^{3/4}(0,T)$.
         It is not difficult to check that the trace operator  $\Gamma:L^2(0,T;H^2(0,L))\cap H^1(0,T;L^2(0,L))\mapsto E$, defined by
         $\Gamma(\xi) := (\xi(\cdot,0),\xi(\cdot,L))$ is surjective, see \cite[p.~18]{lions}.
          Furthermore, there exists a linear continuous mapping $S: E \mapsto L^2(0,T;H^2(0,L))\cap H^1(0,T;L^2(0,L))$ such that $\Gamma\circ S = I_{E}$. 
          Thus, for each $(v_l,v_r)\in E$ we can get $\xi\in L^2(0,T;H^2(0,L))\cap H^1(0,T;L^2(0,L))$ such that
        \[
         \|\xi\|_{L^2(H^2)\cap H^1(L^2)} \leq C( \|v_l\|_{H^{3/4}} + \|v_r\|_{H^{3/4}}),
         \]
          for some $C > 0$.

   	 The following result concerns global existence and uniqueness for viscous Burgers-$\alpha$ systems:
    \begin{prop}\label{re10}
 	 Let $\alpha > 0$, $f\in L^{\infty}((0,T)\times (0,L))$, $y_0\in H^1(0,L)$ and $v_l, v_r \in H^{3/4}(0,T)$ be given. Assume that the following compatibility 
	relations hold:
       \[
      	v_l(0) = y_0(0)\quad\mbox{and}\quad v_r(0) = y_0(L).
       \]
       Then, there exists a unique solution $(y^{\alpha},z^{\alpha})$ to the Burgers-$\alpha$ system:
    \begin{equation}\label{mg4}
    \left\{
    \begin{array}{lll}
    y^\alpha_t - y^\alpha_{xx} + z^\alpha y^\alpha_{x} = f &\mbox{in} & (0,T)\times (0,L),\\
    z^\alpha - \alpha^2 z^\alpha_{xx} = y^\alpha &\mbox{in} & (0,T)\times (0,L),\\
    z^\alpha(\cdot,0) = y^\alpha(\cdot,0) = v_l&\mbox{in} & (0,T),\\
    z^\alpha(\cdot,L) = y^\alpha(\cdot,L) = v_r&\mbox{in} & (0,T),\\
    y^\alpha(0,\cdot) = y_0 &\mbox{in} & (0,L).
    \end{array}
    \right.
    \end{equation}       
       with
       \begin{equation}\label{mg6}
       \left\{
       \begin{array}{l}
         y^{\alpha}\in H^1(0,T;L^2(0,L))\cap L^2(0,T;H^2(0,L))\cap C^0([0,T];H^1(0,L)),\\
         \noalign{\smallskip}\displaystyle
        
        z^{\alpha} \in H^1(0,T;L^2(0,L))\cap L^2(0,T;H^4(0,L))\cap C^0([0,T];H^3(0,L)).
        
       \end{array}
       \right.
       \end{equation}
  
       Let us set $M_T := \|y^0\|_{\infty} + \|v_l\|_{\infty} + \|v_r\|_{\infty} + T\|f\|_{\infty}$. Then, the following estimates holds:
     \begin{equation}\label{mg5}
      \begin{alignedat}{2}
             \|y^\alpha\|_{\infty} \leq &~M_T,\quad \|z^{\alpha}\|_{\infty} \leq M_T,\\
             \noalign{\smallskip}\displaystyle
           \!\!\!  \|y^{\alpha}\|_{H^1(L^2)\cap L^2(H^2)} + \|y^{\alpha}\|_{L^{\infty}(H^1)} \leq&~  Ce^{CM_T^2}\left(\|f\|_2 +\|y_0\|_{H^1} + \|v_l\|_{H^{3/4}} + \|v_r\|_{H^{3/4}} \right)\\
      \noalign{\smallskip}\displaystyle
       \|z^{\alpha}\|_2 +\alpha\|z^{\alpha}_x\|_2 + \alpha^2\|z^{\alpha}_{xx}\|_2 \leq &
       ~Ce^{CM_T^2}\!\left[\|f\|_2 + \|y_0\|_{H^1}\! +(1+\alpha^2)( \|(v_l,v_r)\|_{H^{3/4}\times H^{3/4}})\right].
       \end{alignedat}
        \end{equation}
  \end{prop}
\begin{proof}
	The proof of existence can be reduced to find a fixed-point of an appropriate mapping~$\Lambda_{\alpha}$. Thus, note first that
	there exists $\xi\in L^2(0,T;H^2(0,L))\cap H^1(0,T;L^2(0,L))$ with
\[
              	\xi(\cdot,0) = v_l\,\,\,\mbox{and} \,\,\,\xi(\cdot,L) = v_r\quad \mbox{in}\quad (0,T).
\]
          Accordingly, for each $\bar{y}\in L^{\infty}(0,T;L^\infty(0,L))$ there exists exactly one $z\in L^\infty(0,T;H^2(0,L))\cap L^{\infty}(0,T;L^\infty(0,L))$ with
         \begin{equation*}
         \left\{
         \begin{array}{lll}
         z - \alpha^2 z_{xx} = \bar{y}&\mbox{in} & (0,T)\times (0,L),\\
         z(\cdot,0) = v_l,\,\,z(\cdot,L) = v_r&\mbox{in} & (0,T),\\
         \end{array}
          \right.
           \end{equation*}
      satisfying:
      \begin{equation*}
     \left.
     \begin{array}{l}
     \|z\|^2_2 + 2\alpha^2\|z_x\|^2_2 + \alpha^4\|z_{xx}\|^2_2 \leq C\left(\|\bar{y}\|^2_{2}
         + \|\xi\|^2_{2} + \alpha^2 \|\xi_x\|^2_{2} + \alpha^4\|\xi_{xx}\|^2_{2}\right),\\
         {\|z\|_{L^{\infty}(L^\infty)} \leq \|\bar{y}\|_{L^{\infty}(L^\infty)} + \|v_l\|_{\infty} + \|v_r\|_{\infty}.}
     \end{array}
     \right.
     \end{equation*}
     
      With this $z$, by applying (for instance) the Faedo-Galerkin method, we can easily prove the existence of a  $y$ to the linear parabolic equation
        \begin{equation}\label{mg99}
       \left\{
       \begin{array}{lll}
        y_t - y_{xx} + zy_x = f &\mbox{in} & (0,T)\times (0,L),\\
        y(\cdot,0) = v_l,\,\,y(\cdot,L) = v_r&\mbox{in} & (0,T),\\
        y(0,\cdot) = y_0 &\mbox{in} & (0,L)
       \end{array}
       \right.
       \end{equation}  
       that satisfies
       \begin{equation*}
       \left.
       \begin{array}{l}
       y\in  H^1(0,T;L^2(0,L)) \cap L^2(0,T;H^2(0,L))\cap C^0([0,T];H^1(0,L))
       \end{array}
       \right.
       \end{equation*}
     and 
     \begin{equation}\label{mg2}
     \|y_t\|_{L^2(L^2)} + \|y\|_{L^2(H^2)} + \|y\|_{L^{\infty}(H^1)} \leq C\left(\|y_0\|_{H^1} + \|f\|_{L^2(L^2)} + \|v_l\|_{H^{3/4}} + \|v_r\|_{H^{3/4}}\right)e^{C\|z\|_{\infty}^2}.
     \end{equation} 
     
     Arguing as in the proof of  \cite[Lemma $1$]{Araruna}, we can deduce that the solution to \eqref{mg99} belongs to the space 
     $C^0([0,T];H^1(0,L))$ and, in particular, $y_{L^{\infty}(L^{\infty})} \leq M_T$. Accordingly, we can introduce the bounded closed convex set 
     \[
     K := \{\bar{y}\in L^{\infty}(0,T;L^\infty(0,L))\,:\,                        \|\bar{y}\|_{L^\infty(L^\infty)}\leq M_T\}                   
     \]
     	and the mapping $\Lambda_{\alpha} : K \mapsto K$, with $\Lambda_{\alpha} (\bar{y}) = y$. Obviously, $\Lambda_{\alpha}$ is well-defined and continuous and,
     	moreover, we can see from the estimates in \eqref{mg2} that $G := \Lambda_{\alpha}(K)$ is bounded in $L^{\infty}(0,T;H^1(0,L))$ and $G_t := \{u_t; u\in G\}$ is bounded in 		
	$L^2(0,T;L^2(0,L))$. From classical results of the Aubin-Lions kind (see \cite{simon}), we deduce that $G$ is relatively compact in $L^\infty(0,T;L^\infty(0,L))$. Therefore, 
	by Schauder's fixed point Theorem, $\Lambda_{\alpha}$ has a fixed point in $K$, which obviously implies the existence of a solution to \eqref{mg4}.
     
     We prove now that the solution is unique. 
     Let $(y^{\alpha},z^{\alpha})$ and $(\widehat{y}^{\alpha},\widehat{z}^{\alpha})$ be two solutions to \eqref{mg4} and let us 
     introduce $ u := y^{\alpha} - \widehat{y}^{\alpha}$ and $v := z^{\alpha} - \widehat{z}^{\alpha}$. Then,
     \begin{equation}\label{mg100}
     \left\{
     \begin{array}{lll}
     u_t - u_{xx} + z^{\alpha}u_x = - v\widehat{y}^{\alpha}_x &\mbox{in}& (0,T)\times (0,L),\\
     v - \alpha^2 v_{xx} = u &\mbox{in}& (0,T)\times (0,L),\\
     u(\cdot, 0) = u(\cdot,L) = v(\cdot,0) = v(\cdot,L) = 0 &\mbox{in}& (0,T),\\
     u(0,\cdot) = 0 &\mbox{in}& (0,L).
     \end{array}
     \right.
     \end{equation}   
    Using the fact that $\widehat{y}^{\alpha}\in L^2(0,T;H^2(0,L))\hookrightarrow L^2(0,T;C^1[0,L])$ and multiplying the first equation of the system above by $u$, we get:
    \[
    \left.
    \begin{alignedat}{2}
    \dfrac{1}{2}\dfrac{d}{dt}\|u\|_2^2 + \|u_x\|_2^2 \leq&~ \|z^{\alpha}\|_{\infty}\|u_x\|_2\|u\|_2 + \|\widehat{y}^{\alpha}_x\|_{\infty}\|v\|_2\|u\|_2\\
                                                     \leq &~ \dfrac{1}{2}\|u_x\|_2^2 + \dfrac{\|z^{\alpha}\|_{\infty}^2 }{2} \|u\|_2^2+ \|\widehat{y}^{\alpha}_x\|_{\infty}\|u\|_2^2.
    \end{alignedat}
    \right.
    \]      
    Therefore,
    \[
    \dfrac{d}{dt}\|u\|_2^2 + \|u_x\|_2^2 \leq \left( \|z^{\alpha}\|_{\infty}^2 + 2\|\widehat{y}^{\alpha}_x\|_{\infty}\right)\|u\|_2^2.
    \]
    Since $u(0,\cdot) = 0$, Gronwall's Lemma implies $u \equiv 0$ and, consequently, $v \equiv 0$.
    
    Finally, let us check that $z^{\alpha}$ satisfies the regularity properties in \eqref{mg6}. To get this, let us introduce the function given by
    $$
    h(t,x):={v_l(t)(L-x)+x\,v_r(t)\over L}.
    $$
    Then, we obtain from \eqref{mg4} that $z^\alpha=w^\alpha+h$, where
    $w^\alpha$ solves
        \begin{equation*}
    \left\{
    \begin{array}{lll}
    w^\alpha - \alpha^2 w^\alpha_{xx} = y^\alpha-h &\mbox{in} & (0,T)\times (0,L),\\
    w^\alpha(\cdot,0) = w^\alpha(\cdot,0) = 0&\mbox{in} & (0,T).
    \end{array}
    \right.
    \end{equation*}  
    Consequently, $w^\alpha\in L^2(0,T;H^4(0,L)\cap H^1_0(0,L))\cap C^0([0,T];H^3(0,L)\cap H^1_0(0,L))$ and 
    the estimates are uniform, with respect to $\alpha$, in the space $L^2(0,T;H^2(0,L)\cap H^1_0(0,L))\cap C^0([0,T];H^1_0(0,L))$.
  \end{proof}
  
  \
  
   	Now, let us present a result concerning global existence and uniqueness of a (weak) solution with initial conditions in $L^\infty(0,L)$:
 \begin{prop}\label{re11} 
 	Let $\alpha > 0$, $f\in L^{\infty}((0,T)\times (0,L))$, $y_0\in L^\infty(0,L)$ and $v_l, v_r \in H^{3/4}(0,T)$ be given.
       Then, there exists a unique solution $(y^{\alpha},z^{\alpha})$ to the Burgers-$\alpha$ system:
    \begin{equation}\label{mg44}
    \left\{
    \begin{array}{lll}
    y^\alpha_t - y^\alpha_{xx} + z^\alpha y^\alpha_{x} = f &\mbox{in} & (0,T)\times (0,L),\\
    z^\alpha - \alpha^2 z^\alpha_{xx} = y^\alpha &\mbox{in} & (0,T)\times (0,L),\\
    z^\alpha(\cdot,0) = y^\alpha(\cdot,0) = v_l&\mbox{in} & (0,T),\\
    z^\alpha(\cdot,L) = y^\alpha(\cdot,L) = v_r &\mbox{in} & (0,T),\\
    y^\alpha(0,\cdot) = y_0 &\mbox{in} & (0,L)
    \end{array}
    \right.
    \end{equation}       
        with
       \begin{equation}\label{mg66}
       \left\{
       \begin{array}{l}
         y^{\alpha}\in H^1(0,T;H^{-1}(0,L))\cap L^2(0,T;H^1_0(0,L))\cap  C^0([0,T];L^2(0,L)) \cap L^\infty(0,T;L^\infty(0,L)),\\
         \noalign{\smallskip}\displaystyle
        z^{\alpha} \in H^1(0,T;H^{-1}(0,L))\cap L^2(0,T;H^3(0,L))\cap C^0([0,T];H^2(0,L)).  
       \end{array}
       \right.
       \end{equation}
       Let us set $M_T := \|y^0\|_{\infty} + \|v_l\|_{\infty} + \|v_r\|_{\infty} + T\|f\|_{\infty}$. Then, the following estimates holds:
     \begin{equation}\label{mg55}
      \begin{alignedat}{2}
             \|y^\alpha\|_{\infty} \leq &~M_T, \quad \|z^{\alpha}\|_{\infty} \leq  M_T,\\
             \noalign{\smallskip}\displaystyle
    \!\!\!\!\!  \|y^{\alpha}\|_{H^1(H^{-1})\cap L^2(H^1)} + \|y^{\alpha}\|_{L^{\infty}(L^2)} \leq&~  Ce^{CM_T^2}\left(\|f\|_2 +\|y_0\|_{2} + \|v_l\|_{H^{3/4}} + \|v_r\|_{H^{3/4}} \right),\\
      \noalign{\smallskip}\displaystyle
       \!\!\!\!\!\|z^{\alpha}\|_2 +\alpha\|z^{\alpha}_x\|_2 + \alpha^2\|z^{\alpha}_{xx}\|_2 \leq &
       ~Ce^{CM_T^2}\!\left[\|f\|_2 + \|y_0\|_{2}\! +(1+\alpha^2)( \|(v_l,v_r)\|_{H^{3/4}\times H^{3/4}})\right].
       \end{alignedat}
        \end{equation}
       \end{prop}
  
\begin{proof}
         For any $\bar{y}\in L^2(0,T;L^\infty(0,L))$, there exists a unique solution to the elliptic problem
         \begin{equation*}
         \left\{
         \begin{array}{lll}
         z - \alpha^2 z_{xx} = \bar{y}&\mbox{in} & (0,T)\times (0,L),\\
         z(\cdot,0) = v_l,\,\,z(\cdot,L) = v_r&\mbox{in} & (0,T),\\
         \end{array}
          \right.
           \end{equation*}
      furthermore satisfying
      \begin{equation*}
     \left.
     \begin{array}{l}
     \|z\|^2_2 + 2\alpha^2\|z_x\|^2_2 + \alpha^4\|z_{xx}\|^2_2 \leq C\left(\|\bar{y}\|^2_{2}
         + \|\xi\|^2_{2} + \alpha^2 \|\xi_x\|^2_{2} + \alpha^4\|\xi_{xx}\|^2_{2}\right),\\
         \|z\|_{L^{2}(L^\infty)} \leq \|\bar{y}\|_{L^{2}(L^\infty)} + \|v_l\|_{2} + \|v_r\|_{2}.
 \end{array}
     \right.
     \end{equation*}
     
      With this $z$, we solve the linear problem
        \begin{equation}\label{mg1}
       \left\{
       \begin{array}{lll}
        y_t - y_{xx} + zy_x = f &\mbox{in} & (0,T)\times (0,L),\\
        y(\cdot,0) = v_l,\,\,y(\cdot,L) = v_r&\mbox{in} & (0,T),\\
        y(0,\cdot) = y_0 &\mbox{in} & (0,L)
       \end{array}
       \right.
       \end{equation}  
       and we find a solution $y$ that satisfies
     \begin{equation*}\label{mg222}
     \!\|y_t\|_{L^2(H^{-1})} + \|y\|_{L^2(H^1_0)} + \|y\|_{L^{\infty}(L^2)} \leq C\left(\|y_0\|_{2} + \|f\|_2 + \|v_l\|_{H^{3/4}} + \|v_r\|_{H^{3/4}}\right)e^{C\|z\|_{L^2(L^\infty)}^2}.
     \end{equation*} 
     
     Again, as in the proof of  \cite[Lemma $1$]{Araruna}, we can deduce that the solution to \eqref{mg1} satisfies 
      $$\|y\|_{L^2(L^\infty)}\leq T^{1/2}M_T.$$
       Let us introduce the set 
       \[
        K := \{\bar{y}\in L^{2}(0,T;L^\infty(0,L))\,:\,                        \|\bar{y}\|_{L^2(L^\infty)}\leq T^{1/2}M_T\}                   
     \]
     	and the mapping $\Lambda_{\alpha} : K \mapsto K$ with $\Lambda_{\alpha} (\bar{y}) = y$. Then, arguing as in the proof of Proposition~\ref{re10}, it is not difficult to prove that $\Lambda_{\alpha}$ possesses a fixed-point in $K$.
     	
     Finally, in order to prove uniqueness, we consider to solutions $ u := y^{\alpha} - \widehat{y}^{\alpha}$ and $v := z^{\alpha} - \widehat{z}^{\alpha}$ and we get \eqref{mg100}. Then, multiplying the first equation of \eqref{mg100} by $u$, we easily get the differential inequality
     
    \[
    \dfrac{d}{dt}\|u(t,\cdot)\|_2^2 + \|u_x(t,\cdot)\|_2^2 \leq \left( \|z^{\alpha}(t,\cdot)\|_{\infty}^2 + {2C\|\widehat{y}^{\alpha}_x(t,\cdot)\|_{2}\over \alpha}\right)\|u(t,\cdot)\|_2^2.
    \]
    Since $u(0,\cdot) \equiv 0$, Gronwall's Lemma implies $u \equiv 0$ and, consequently, $v \equiv 0$. 
    
    Let $(y^\alpha,z^\alpha)$ be the solution to \eqref{mg44}. From \eqref{mg44} and the fact that $y\in L^\infty(0,T;L^\infty(0,L))$, the maximum principle implies that
    $$
    \|z^\alpha\|_{L^\infty(L^\infty)}\leq \|y^\alpha\|_{L^\infty(L^\infty)}\leq M_T.
    $$
    
    This ends the proof.
  \end{proof}

\section{Controllability of the inviscid Burgers-$\alpha$ system} \label{sec:inviscid}

	In this section we present a proof of the global exact controllability property of the inviscid Burgers-$\alpha$ system. We split the proof in two parts: $(i)$ a local null controllability result; $(ii)$ an argument based on a time scale-invariance and reversibility in time that leads to the desired global result.


\subsection{Local null controllability}

 We have the following result:	
\begin{prop}\label{es17}
	Let $T,L, \alpha > 0$ be given. Then, there exist $\delta > 0$ and $C > 0$ (both independent of $\alpha$) such that the following property holds: 
	for each $y_0\in C^1([0,L])$ with $\|y_0\|_{C^1([0,L])}\leq\delta$, there exist  $p^{\alpha} \in C^0([0,T])$ with $p^\alpha(T)=0$, $v_l^{\alpha}, v_r^{\alpha} \in C^1([0,T])$
	and associated states $(y^{\alpha},z^{\alpha}) \in C^1([0,T]\times [0,L];\mathbb{R}^2)$ satisfying \eqref{es14},
	\begin{equation*}\label{es15}
		y^{\alpha}(T,\cdot) = z^{\alpha}(T,\cdot)= 0 \quad\mbox{in} \quad (0,L)
	\end{equation*}
	 and
\begin{equation*}\label{v-unif}
	\|p^{\alpha}\|_{C^0([0,T])}+\|(v_l^{\alpha},v_r^{\alpha})\|_{C^1([0,T];\mathbb{R}^2)}\leq C \quad \forall \alpha>0.
\end{equation*}
\end{prop}

	The proof is easy by applying the return method, see \cite{Chapouly, Coron, Coron1,Glass}. It relies on a linearization process in combination with a fixed-point argument: 
	$(i)$ first, we need to find a ``good" trajectory (a particular solution for the nonlinear system) steering $0$ to $0$ such that the linearization around it is controllable; 
	$(ii)$ then, we must recover (for instance by a fixed-point argument) the exact controllability result, at least locally, for the nonlinear system.
	
	In our case, it is not difficult to verify that the linearization around zero is not controllable. Accordingly, we  build an appropriate nontrivial trajectory connecting 
	$(0,0)$ to $(0,0)$.
	
	To this purpose, let us introduce the set 
           \begin{equation*}
            \Lambda_{L,T,k}:=\left\{\lambda\in C^k([0,T];[0,\infty))\,:\, \|\lambda\|_{L^1(0,T)} > L\right\}.
            \end{equation*}
		
	Let us consider the couple $(\widehat{y}(x,t),\widehat{z}(x,t)):=(\lambda(t),\lambda(t))$ and the triplet $(\widehat{p}(t),\widehat{v}_l(t),\widehat{v}_r(t)) := (\lambda'(t),\lambda(t),\lambda(t))$, 
	with $\lambda\in  \Lambda_{L,T,1}$ and $\mbox{supp}\,\, \lambda\subset (0,T)$. Note that $(\widehat{y},\widehat{z})$ is a particular solution to \eqref{es14}, associated to the control 
	$(\widehat{p},\widehat{v}_l,\widehat{v}_r)$. We have the following general controllability result:	
       \begin{prop}\label{sp}
	Let $T ,\,L > 0$ be given and assume that $\lambda\in  \Lambda_{L,T,0}$. Then, for any $\alpha>0$ and any $y_0\in C^1([0,L])$, there exists $(y,z)\in C^1([0,T]\times [0,L];\mathbb{R}^2)$ such that
          \begin{equation}\label{eq:chapouly}
          \left\{
          \begin{array}{lcl}
          y_t + \lambda(t) y_x = 0								&\mbox{in}&(0,T)\times (0,L),\\
         \noalign{\smallskip}\displaystyle
          z -\alpha^2 z_{xx} = y						&\mbox{in}& (0,T)\times (0,L),\\
        \noalign{\smallskip}\displaystyle
          z(\cdot,0) = y(\cdot,0),\quad z(\cdot,L) = y(\cdot,L) 	&\mbox{in}& (0,T),\\  
        \noalign{\smallskip}\displaystyle
          y(0,\cdot) = y_0							&\mbox{in}& (0,L),\\
         \noalign{\smallskip}\displaystyle
          y(T,\cdot) = 0								&\mbox{in}&(0,L).
         \end{array}
         \right.
         \end{equation} 
          \end{prop}
          
	For the proof, it suffices to use  \cite[Proposition $8$]{Chapouly} to find $y\in C^1([0,T]\times [0,L])$ satisfying \eqref{eq:chapouly}$_{1}$,\eqref{eq:chapouly}$_{4}$ and \eqref{eq:chapouly}$_5$ 
	and then solve the elliptic problem \eqref{eq:chapouly}$_{2}$-\eqref{eq:chapouly}$_{3}$ to construct $z\in C^1([0,T]\times [0,L])$.
	 
	Thanks to Proposition \ref{sp}, one may expect that the null controllability for the nonlinear system  \eqref{es14} holds. Indeed, we have the following result 
	from which Proposition \ref{es17} is an immediate consequence:

\begin{prop}\label{sp2}
	Let $T,\,L > 0$ be given and assume that $\lambda\in  \Lambda_{L,T,0}$. Then, there exist $\delta > 0$ and $C > 0$ (both independent of $\alpha$) such that,
	for any ~$y_0\in C^1([0,L])$ with $\|y_0\|_{C^1([0,L])}\leq \delta$ and any~$\alpha > 0$, there exist  $(v_l,v_r)\in C^1([0,T];\mathbb{R}^2)$ 
	and an associated state $(y,z)\in C^1([0,T]\times [0,L];\mathbb{R}^2)$ satisfying
         \begin{equation}\label{sp1}
         	\left\{
          		\begin{array}{lcl}
          			y_t + (\lambda(t) + z)y_x = 0			&\mbox{in}& 		(0,T)\times (0,L),\\
         		\noalign{\smallskip}\displaystyle	
			z - \alpha^2 z_{xx} = y 				&\mbox{in}& 		(0,T)\times (0,L),\\
          	\noalign{\smallskip}\displaystyle	
			y(\cdot,0) = z(\cdot,0) = v_l 			&\mbox{in}& 		(0,T),\\
         	\noalign{\smallskip}\displaystyle	
			y(\cdot,L) = z(\cdot,L) = v_r 			&\mbox{in}& 		(0,T),\\    
         	\noalign{\smallskip}\displaystyle	
			y(0,\cdot) = y_0						&\mbox{in}&		(0,L),\\
        \noalign{\smallskip}\displaystyle 	
			y(T,\cdot) = 0						&\mbox{in}& 		(0,L)
         		 \end{array}
         	 \right.
         \end{equation}
         and
\begin{equation*}\label{v-unif}
	\|y\|_{C^1([0,T]\times [0,L])}\leq C\|y_0\|_{C^1([0,L])}\quad \forall \alpha>0.
\end{equation*} 
\end{prop}

\begin{proof} 
	We will reformulate the null controllability problem as a fixed-point equation. To do this, we will first introduce some auxiliar functions and establish some helpful results. Thus, to any given 
	$h \in  C^0([0,T];C^0([0,L]))\cap L^\infty\left(0,T; C^{0,1}([0,L])\right)$ we can associated the unique solution to  the time-dependent problem
           \begin{equation}\label{sp3}
           \left\{
           \begin{array}{lll}
           z - \alpha^2 z_{xx} = h 		&\mbox{in}& (0,T)\times (0,L),\\
        \noalign{\smallskip}\displaystyle
           z(\cdot,0) = h(\cdot,0) 	&\mbox{in}& (0,T),\\
        \noalign{\smallskip}\displaystyle
           z(\cdot,L) = h(\cdot,L)  	&\mbox{in}& (0,T).
           \end{array}
           \right.
           \end{equation}
           From the {\it maximum principle} for elliptic equations, we get
           \begin{equation}\label{sp4}
           \|z\|_{C^0([0,T];C^0([0,L]))}\leq   \|h\|_{C^0([0,T];C^0([0,L]))}
           \end{equation}
            and
            \begin{equation*}\label{sp41}
           \|z_x\|_{C^0([0,T];L^\infty(0,L))}\leq   \|h_x\|_{L^\infty(0,T;L^\infty(0,L))}.
           \end{equation*}          
     Since $\lambda\in  \Lambda_{L,T}$,  we can find $\eta \in (0,L/2)$ such that
     \begin{equation}
     \int_{0}^{T} \lambda(s)\, ds > L + 2\eta. \label{sp7}
     \end{equation}
     
     Now, we consider the following extension of $z$ to the closed interval $[-\eta, L + \eta]$:
     \begin{equation*}\label{sp11}
     z^\eta(t,x):=
     \left\{
     \begin{array}{ll}
     5z(t,-x) - 20z\left(t,-{x\over2}\right) + 16z\left(t,-{x\over4}\right),&(t,x)\in [0,T]\times [-\eta,0],\\ 
 \noalign{\smallskip}
     z(t,x)& (t,x),\in [0,T]\times [0,L],\\
 \noalign{\smallskip}     5z(t,2L-x) - 20z\left(t,{3L - x\over2}\right) + 16z\left(t,{5L - x\over4}\right),& (t,x)\in [0,T]\times [L,L+ \eta].
     \end{array}
     \right.
     \end{equation*}        
         It is not difficult to check that $z^\eta\in C^0([0,T];C^2([-\eta, L + \eta]))$ and there exists $C_1 > 0$ $($independent of $\alpha$$)$ such that
  \begin{equation}\label{es22}
  \|z^\eta\|_{C^0([0,T];C^{0,1}([-\eta, L + \eta]))}\leq C_1\|z\|_{C^0([0,T];C^{1}([0,L]))}.
  \end{equation}

        Then, let $\chi$ be given, with $\chi\in C_0^{\infty}(-\eta/2, L + \eta/2)$, $\chi = 1$ in $[0,L]$ and $0\leq\chi\leq 1$.


	This way, we can introduce $z^*\in C^0([0,T];C^2(\mathbb{R}))$, with
           \begin{equation}
           z^*(t,x) = 
           \left\{
           \begin{array}{ll}
           \chi(x)z^\eta(t,x),& (t,x)\in [0,T]\times[-\eta, L + \eta],\\
         \noalign{\smallskip}\displaystyle	
           0,& (t,x)\in [0,T]\times(\mathbb{R}\setminus[-\eta, L + \eta]).   \label{mg17}
           \end{array}
           \right.
           \end{equation}
    and, using \eqref{es22}  and \eqref{mg17}, we see that
          \begin{equation}\label{sp21}
          \|z^*\|_{C^0([0,T];C^{0,1}_b(\mathbb{R}))} \leq C_2\|z\|_{C^0([0,T];C^{0,1}([0,L]))},
          \end{equation}
          for some $C_2 > 0$, again independent of $\alpha$.
   
   	Let us set $R := {\eta\over C_2T}$ and let us assume from now on that
          \begin{equation}
          \|h\|_{C^0([0,T];C^{0,1}([0,L]))} \leq R. \label{sp13} 
          \end{equation}
	Then, it follows from \eqref{sp4}, \eqref{sp21} and \eqref{sp13} that
          \begin{equation}\label{sp141}
          \|z^*\|_{C^0([0,T];C^{1}_b(\mathbb{R}))}\leq {\eta\over T}. 
          \end{equation}

    Let $\phi^*$ be the flow associated with the ordinary differential equation $\xi' = \lambda(t) + z^*(t,\xi)$, that is,  the solution to
        \begin{equation}\label{sp16}
        \left\{
        \begin{array}{ll}
          \displaystyle\frac{\partial\phi^*}{\partial t}(s;t,x) = \lambda(t) + z^*(t, \phi^*(s;t,x)),&\\
       \noalign{\smallskip}\displaystyle	
        \phi^*(s;s,x) = x. &
        \end{array}
        \right.
        \end{equation}

\begin{claim}\label{Claim1} The function $\phi^* = \phi^*(s;t,x)$ is well-defined for any $(t,x)\in [0,T]\times \mathbb{R}$ and $s\in [0,T]$.
\end{claim}

\begin{proof} Let $\phi : [0,T]\times [0,T]\times\mathbb{R}\mapsto\mathbb{R}$ be the flow associated to the ODE $\xi' = \lambda(t)$. 
Then, for every $(s,t,x)$ we get from \eqref{sp16} that

        \begin{equation*}
        \left.
        \begin{array}{lll}
        |\phi^*(s;t,x) - \phi(s;t,x)|&=& \left|\displaystyle\int_{s}^{t}  \left(\frac{\partial\phi^*}{\partial \tau}(s;\tau,x) - \frac{\partial\phi}{\partial\tau}(s;\tau,x) \right)\,d\tau\right|\\
        \noalign{\smallskip}\displaystyle
        &\leq & \displaystyle\int_{s}^{t}\left|\frac{\partial\phi^*}{\partial\tau}(s;\tau,x) -  \frac{\partial\phi}{\partial\tau}(s;\tau,x)\right|\,d\tau\\
        \noalign{\smallskip}\displaystyle
        &= & \displaystyle\int_{s}^{t}|z^*(\tau,\phi^*(s;\tau,x))|\,d\tau\\
       \noalign{\smallskip}\displaystyle
        &\leq & T\|z^*\|_{C^0([0,T];C^0(\mathbb{R}))}.
        \end{array}
        \right.
       \end{equation*}
        Hence, for any $(s,t,x)$ such that $\phi^*(s;t,x)$ is well-defined, one has
        \begin{equation}\label{sp17}
        |\phi^*(s;t,x) - \phi(s;t,x)| \leq \eta.  
        \end{equation}
This and the fact that $\phi(s;t,x)$ is well-defined for all $(s,t,x)\in [0,T]\times [0,L]\times\mathbb{R}$ lead to the desired conclusion.
\end{proof}

\

Let $y_0\in C^1([0,L])$ be given and let us introduce $y_0^\eta\in C^1\left([-\eta,L + \eta]\right)$ with

      \begin{equation*}\label{sp8}
      y_0^\eta(x) = 
      \left\{
      \begin{array}{ll}
      -y^0(-x) + 2y^0(0),& x\in [-\eta,0],\\
      \noalign{\smallskip}\displaystyle	
      y^0(x),& x\in [0,L],\\ 
     \noalign{\smallskip}\displaystyle	
      -y^0(2L-x) + 2y^0(L),& x\in [L,L+\eta]
      \end{array}
      \right.
      \end{equation*}
     and 
      \begin{equation*}
      y^{*}_0(x) = 
      \left\{
      \begin{array}{ll}
      \chi(x)y_0^\eta(x),& x\in [-\eta, L + \eta],\\
     \noalign{\smallskip}\displaystyle	
      0,& x\in\mathbb{R}\setminus [-\eta, L + \eta].  \label{sp18}
      \end{array}
      \right.
      \end{equation*} 
      Then, it is easy to see that $y_0^{*}$ is an extension of $y_0$ to the whole real line and 
      \begin{equation}\label{sp10}
      \|y^{*}_0\|_{C^{1}_b(\mathbb{R})}\leq C_3\|y_0\|_{C^{1}([0,L])}.
      \end{equation}
      for some $C_3 > 0$.
      
	Let us set $y\in C^1([0,T]\times\mathbb{R})$, with
        \begin{equation}\label{sp19}
         y(t,\bar x) := y_0^{*}(\phi^*(t;0,\bar x))\quad \forall (t,\bar x)\in [0,T]\times\mathbb{R}.
        \end{equation}
     Then, we have the following :
     
      \begin{claim}\label{es27}
	The function $y$ satisfies:
       \begin{equation}\label{eq:transport}
       \left\{
       \begin{array}{lll}
       y_t + (\lambda(t) + z^*(t,x))y_x = 0 &\mbox{in}& (0,T)\times\mathbb{R},\\
          \noalign{\smallskip}\displaystyle	
       y(0,\cdot) = y_0^{*} &\mbox{in} & \mathbb{R},\\
           \noalign{\smallskip}\displaystyle	
        y(T,\cdot) = 0 &\mbox{in} &[0,L].
       \end{array}
       \right.
       \end{equation}
      \end{claim}
\begin{proof}
	For any $t\in [0,T]$, $\phi^*(0;t,\cdot):\mathbb{R}\to\mathbb{R}$ is a diffeomorphism  and \eqref{sp19} is equivalent to $y(t,\phi^*(0,t,x)))\equiv y^{*}_{0}(x)$. 
	Then, for each $x\in \mathbb{R}$, we deduce that 
       \begin{equation*}
       \left.
       \begin{array}{rll}
       & y_t(t,\phi^*(0;t,x))+\left[\lambda(t) + z^*(t,\phi^*(0;t,x))\right]y_x(t,\phi^*(0;t,x)) = 0.
       \end{array}
      \right.
      \end{equation*}
      Using \eqref{sp19} and $\eqref{sp16}_2$,  we get
      \[
       y(0,x) =  y^{*}_0(x)\quad\forall\,\, x\in \mathbb{R}.
       \]
    Moreover, it is not difficult to see that, for any $0 < \eta < L/2$ such that (\ref{sp7}) holds, the flow associated to the ODE $\xi' = \lambda(t)$ satisfies $\phi(T;0,L) < -2\eta$
    and we obtain from \eqref{sp17} that $\phi^*(T;0,L) < -\eta$. 
    Since $\phi^*(s;t,\cdot)$ is increasing for any $s,t\in [0,T]$, we see that
          \begin{equation*}
          \phi^*(T;0,x) < -\eta \quad \forall\,\, x\in (-\infty,L].
          \end{equation*}
	This inequality, together with the fact that $$
		\mbox{supp}~y^{*}_0\subset\left[-\eta, L + \eta\right],
	$$
	implies $y(T,\cdot) = 0$ in $[0,L]$.
\end{proof}
      
      \
      
	An immediate consequence of \eqref{sp141}, $C^1$ estimates for \eqref{sp16}, \eqref{sp10}  and \eqref{sp19} is that 
\begin{equation*}\label{est:init_data}
		\|y\|_{C^1([0,T]\times [0,L])}\leq C_4 \|y_0\|_{C^{1}([0,L])},
\end{equation*}
	for a positive constant $C_4$ depending on $R$ but independent of $\alpha$.
	Taking $y_0\in C^1([0,L])$ such that $$\|y_0\|_{C^{0,1}([0,L])}\leq R/C_4,$$ 
	we have that $\|y\|_{C^0([0,T];C^{0,1}([0,L]))}\leq R$ and 
     we can therefore introduce the mapping $\mathcal{F} : B_R \mapsto B_R$, where $B_R$ is the closed ball of radius $R$ in the space $C^0([0,T];C^0([0,L]))\cap L^{\infty}(0,T;C^{0,1}([0,L]))$ and, for each $h\in B_R$, $y = \mathcal{F}(h)$ is given by \eqref{sp19}.

    
    Thanks to \eqref{sp19}, we have that $ \mathcal{F}(B_R)\subset C^1([0,T]\times [0,L])$. Moreover, the following holds: 
     \begin{claim}\label{es29}
     There exists a positive constant $C$ that depends on $\|y_0\|_{C^1([0,L])},L,R$ and $T$, such that, for any $m\geq 1$ and any $h^1,h^2\in B_R$, one has 
     \begin{align*}
     \|(\mathcal{F}^m(h^1) - \mathcal{F}^m(h^2))(t,\cdot)\|_{C^0([0,L])} \leq \dfrac{(Ct)^m}{m!} \|h^1 - h^2\|_{C^0([0,T];C^0([0,L]))}\,\,\mbox{in}\,\, [0,T].
     \end{align*}
     \end{claim}
     
     \begin{proof} The proof relies on an induction argument. Let $h^i\in B_R$ be given for $i=1,2$. Then, let us consider the functions $z^{i,*}$ and $y^i$, respectively  given by \eqref{mg17} and \eqref{sp19} and set $y := y^1 - y^2$ and $z^* := z^{1,*} - z^{2,*}$. By Claim \ref{es27}, we have
       \[
       y_t + (\lambda + z^{1,*})y_x = -z^* y^2_{x}\quad \mbox{in}\quad (0,T)\times\mathbb{R},
       \]
        whence, from Proposition \ref{es28},
        \begin{align*}
        \dfrac{d}{dt^+}\|y(t,\cdot)\|_{C^0_b(\mathbb{R})} & \leq \|z^*(t,\cdot)y^2_x(t,\cdot)\|_{C^0_b(\mathbb{R})}.
        \end{align*}
       Therefore, integrating from $0$ to $t$ and using that $y^2_{x} \in C^0([0,T];C^0_b(\mathbb{R}))$ and the maximum principle for elliptic PDE's, we find a positive constant $C$ depending on  $\|y_0\|_{C^1([0,L])},L,R$ and $T$, such that
\begin{align*}
        	\|y(t,\cdot)\|_{C^0_b(\mathbb{R})} &\leq C\displaystyle\int_{0}^{t}\|z^*(\tau,\cdot)\|_{C^0_b(\mathbb{R})}\,d\tau\\
	&\leq C\displaystyle\int_{0}^{t}\|z^1(\tau,\cdot) - z^2(\tau,\cdot)\|_{C^0([0,L])}\,d\tau\\
	&\leq C\displaystyle\int_{0}^{t}\|h^1(\tau,\cdot) - h^2(\tau,\cdot)\|_{C^0([0,L])}\,d\tau.
\end{align*} 
         
         It follows that
        \begin{align}\label{es30}
        \|(\mathcal{F}(h^1) - \mathcal{F}(h^2))(t,\cdot)\|_{C^0([0,L])}& \leq Ct\|h^1 - h^2\|_{C^0([0,T];C^0([0,L]))} 
        \end{align}
       and the result is true for $m = 1$. 
        
        Now, assume that the claim is true for a fixed $m$ and let us prove that it holds also for $m + 1$. \\
        Performing computations similar to those above, we get
         \begin{align*}
         \|(\mathcal{F}^{m+1}(h^1) - \mathcal{F}^{m+1}(h^2))(t,\cdot)\|_{C^0([0,L])} & \leq C\displaystyle\int_{0}^{t}\|(\mathcal{F}^{m}(h^1) - \mathcal{F}^{m}(h^2))(\tau,\cdot)\|_{C^0([0,L])}\,\,d\tau. 
         \end{align*}
         where $C$ is the same positive constant in \eqref{es30}.
                
         Using the induction hypothesis, we deduce that
         \begin{align*}
         \|(\mathcal{F}^{m+1}(h^1) - \mathcal{F}(h^2)^{m+1})(t,\cdot)\|_{C^0([0,L])} & \leq C\|h^1 - h^2\|_{C^0([0,T];C^0([0,L]))} \displaystyle\int_{0}^{t}\dfrac{(C\tau)^m}{m!}\,\,d\tau\\
         & = \dfrac{(Ct)^{m+1}}{(m + 1)!}\|h^1 - h^2\|_{C^0([0,T];C^0([0,L]))}.
         \end{align*}
         Therefore, the result is also true for $m + 1$ and the proof is done.                   
     \end{proof}
     
     \
     
      Let $\widetilde B_R$ be the closure of $B_R$ with the norm of $C^0([0,T];C^0([0,L]))$ and  let $\widetilde{\mathcal{F}}$ 
      be the unique continuous extension of  $\mathcal{F}$ to $\widetilde B_R$.
\begin{claim}\label{es30}
	The extension $\widetilde{\mathcal{F}}$ satisfies: 
	\begin{align*}
		\widetilde{\mathcal{F}}(\widetilde B_R)\subset  B_R.
	\end{align*}
\end{claim}
\begin{proof} Let $h$ be a function in  $\widetilde B_R$. Then:
	\begin{itemize}
		\item $\widetilde{\mathcal{F}}(h)\in C^1([0,T]\times [0,L])$ and solves \eqref{eq:transport}. Indeed, let $(h_n)_{n\in \mathbb{N}}$ be a sequence in  $B_R$ such
	that $h_n\to h$ in $C^0([0,T];C^0([0,L]))$. Then, the corresponding elliptic solutions to \eqref{sp3} and 
	the associated flows satisfy 
	$$
		z_n\to z~~\text{in}~~C^0([0,T];C^2([0,L]))\quad \text{and}\quad \Phi_n \to \Phi ~~\text{in}~~C^0([0,T]\times [0,T]\times  \mathbb{R}).
	$$
	Moreover, since the $ \Phi_n, \Phi \in C^1([0,T]\times [0,T]\times  \mathbb{R})$, verify the corresponding functions defined by \eqref{sp19}
	belong to $C^1([0,T]\times  \mathbb{R})$ and verify
	$$
		y_n\to y~~\text{in}~~C^0([0,T];C^0([0,L])).
	$$	
	Therefore, $\widetilde{\mathcal{F}}(h)=y\in C^1([0,T]\times [0,L])$.
		\item $\widetilde{\mathcal{F}}(h)\in B_R$. In fact, notice that, by definition of $\widetilde{\mathcal{F}}$, there exists a sequence $(h_n)_{n\in \mathbb{N}}\subset B_R$ such that
		$h_n\to h$ in $C^0([0,T];C^0([0,L]))$ and 
		$$
			\mathcal{F}(h_n)\to \widetilde{\mathcal{F}}(h)\quad \text{in}\quad C^0([0,T];C^0([0,L])).
		$$
	On the other hand, it is not difficult to prove that
	$$
	\begin{alignedat}{2}
		\|\widetilde{\mathcal{F}}(h)(t,\cdot)\|_{C^0([0,L])}+\|\widetilde{\mathcal{F}}(h)(t,\cdot)\|_{C^{0,1}([0,L])}
		\leq &~\|\mathcal{F}(h_n)(t,\cdot)\|_{C^0([0,L])}+\|\mathcal{F}(h_n)(t,\cdot)\|_{C^{0,1}([0,L])} \\
		&~4\|\mathcal{F}(h_n)(t,\cdot)-\widetilde{\mathcal{F}}(h)(t,\cdot)\|_{C^0([0,L])}.
	\end{alignedat}
	$$	
		Therefore, using the fact that $\widetilde{\mathcal{F}}(h)\in C^1([0,T]\times [0,L])$, we certainly have that  $\widetilde{\mathcal{F}}(h)\in B_R$.
	\end{itemize}
\end{proof}

\
      
     It follows from Claim \ref{es29} that $\widetilde{\mathcal{F}}^m$ is a contraction for $m$ large enough. Then, from Banach Fixed-Point Theorem \ref{fixed-point}, 
     $\widetilde{\mathcal{F}}$ possesses a unique fixed-point $y\in \widetilde B_R$. Finally, taking into account Claim \ref{es30}, the proof of Proposition \ref{sp2} is achieved.   
\end{proof}

\begin{remark}\label{rm:1}
  	\textrm{Let $T,\,L > 0$, assume that $\lambda\in  \Lambda_{L,T,1}$ and
   consider the Banach space
	$$
	\mathcal{X}= C^0([0,T];C^1([0,L]))\cap C^1([0,T];C^0([0,L]))
	\cap L^{\infty}(0,T; C^{1,1}([0,L])) 
	$$
	If $y_0\in C^2([0,L])$ is small enough, then the fixed-point mapping $\mathcal{F}$ can be defined in a closed ball of $\mathcal{X}$  centered at zero of radius $R>0$. 
	Then, one applies  Banach Fixed-Point Theorem in the closure of this ball with the norm of $C^0([0,T];C^1([0,L]))\cap C^1([0,T];C^0([0,L]))$. 
   Performing similar computations of Proposition \ref{sp2}, one can deduce that there exists $\delta > 0$ (independent of $\alpha$) such that,
	for any ~$y_0\in C^2([0,L])$ with $\|y_0\|_{C^2([0,L])}\leq \delta$, there exists a solution $y\in C^2([0,T]\times [0,L])$ to
	\eqref{sp1}, satisfying
   \begin{equation}\label{eq:c2}
         	\|y\|_{C^2([0,T]\times[0,L]))}\leq C\|y_0\|_{C^2([0,L])} \quad \forall \alpha>0,
	\end{equation}
	for a constant $C>0$ that is independent of $\alpha$.}

\end{remark}

	\

\subsection{Global exact controllability}

	In order to prove Theorem \ref{thm1}, we have to use scaling arguments and the time-reversibility of the inviscid Burgers-$\alpha$ system. 
	Thus, let $T,L > 0$ be given, let us consider initial and final states $y_0, y_T\in C^1([0,L])$, let $\delta > 0$ be given by Proposition \ref{es17} 
	and let $\gamma_0, \gamma_T\in (0,1)$ be such that $\gamma_0 < \gamma_T$,           
	         \begin{align*}
            \|\gamma_0 y_0\|_{C^1([0,L])} \leq \delta\,\,\,\mbox{and}\,\,\, \|(1-\gamma_T) y_T\|_{C^1([0,L])} \leq \delta.
            \end{align*}
             
Then, by Proposition \ref{es17}, there exist distributed controls $\widetilde p$, $\widehat{p}$ in $C^0_c((0,T))$, boundary controls $(\widetilde v_l,\widetilde v_r)$, $(\widehat{v}_l,\widehat{v}_r)$ in $C^1([0,T])$ and associated states 
$(\widetilde y,\widetilde z)$, $(\widehat{y},\widehat{z})$ in $C^1([0,T]\times[0,L])$ such that
    \begin{equation}\label{es23}
    \left\{ 
    \begin{array}{lll}
    \widetilde y_t +\widetilde z\,\widetilde y_x =\widetilde p(t)&\mbox{in}&(0,T)\times (0,L),\\
   \widetilde z - \alpha^2 \widetilde z_{xx} =\widetilde y &\mbox{in}& (0,T)\times (0,L),\\
   \widetilde z(\cdot,0) = \widetilde y(\cdot,0) =\widetilde v_l &\mbox{in}& (0,T),\\
   \widetilde z(\cdot,L) = \widetilde y(\cdot,L) =\widetilde v_r &\mbox{in}& (0,T),\\
   \widetilde y(0,\cdot) = \gamma_0 y_0(x) &\mbox{in}& (0,L),\\
   \widetilde y(T,\cdot) = 0 &\mbox{in} & (0,L)
    \end{array}
    \right.
    \end{equation} 
and
   \begin{equation}\label{es24}
   \left\{
   \begin{array}{lll}
   \widehat{y}_t + \widehat{z}\widehat{y}_x = \widehat{p}(t)&\mbox{in}&(0,T)\times (0,L),\\
   \widehat{z} - \alpha^2 \widehat{z}_{xx} = \widehat{y} &\mbox{in}& [0,T]\times [0,L],\\
   \widehat{z}(\cdot,0) = \widehat{y}(\cdot,0) = \widehat{v}_l &\mbox{in}& (0,T),\\
   \widehat{z}(\cdot,L) = \widehat{y}(\cdot,L) = \widehat{v}_r &\mbox{in}& (0,T),\\
   \widehat{y}(0,\cdot) = (1 -\gamma_T) y_T &\mbox{in}& (0,L),\\
   \widehat{y}(T,\cdot) = 0 &\mbox{in} & (0,L).
   \end{array}
   \right.
   \end{equation}
    
Using $\eqref{es23}, \eqref{es24}$ and the facts that $\widetilde p(T) = \widehat{p}(T) = 0$ and $\gamma_0\thicksim 0$ and $\gamma_T\thicksim 1$, we can introduce the functions $Y, Z : [0,T]\times [0,L]\mapsto\mathbb{R}$ and $P, V_l, V_r : [0,T]\mapsto\mathbb{R}$, given by 
     \begin{equation*}
     Y(t,x) := 
     \left\{
     \begin{array}{ll}
     \gamma_0^{-1}\, \widetilde y\left(t\,\gamma_0^{-1},x\right)& (t,x)\in \left[0,\gamma_0 T\right]\times   [0,L],\\
   \noalign{\smallskip}\displaystyle	
    0 & (t,x)\in \left[\gamma_0 T,\gamma_1 T\right]\times [0,L],\\
    \noalign{\smallskip}\displaystyle
    \dfrac{1}{1 - \gamma_T}\widehat{y}\left(\dfrac{T - t}{1 - \gamma_1}, L - x\right) & (t,x)\in     [\gamma_1 T,T]\times [0,L],
    \end{array}
    \right.
    \end{equation*}
     \begin{equation*}
     Z(t,x) :=
    \left\{
    \begin{array}{ll}
     \gamma_0^{-1}\, \widetilde z\left(t\,\gamma_0^{-1},x\right)& (t,x)\in \left[0,\gamma_0 T\right]\times  [0,L],\\
  \noalign{\smallskip}\displaystyle	
    0 & (t,x)\in \left[\gamma_0 T,\gamma_1 T\right]\times [0,L],\\
     \noalign{\smallskip}\displaystyle
    \dfrac{1}{1 - \gamma_T}\widehat{z}\left(\dfrac{T - t}{1 - \gamma_T}, L - x\right) & (t,x)\in    [\gamma_1 T,T]\times [0,L],
    \end{array}
    \right.
    \end{equation*}
     \begin{equation*}
P(t) :=
\left\{
\begin{array}{ll}
 \gamma_0^{-2}\, \widetilde p\left(t\,\gamma_0^{-1}\right)& t\in \left[0,\gamma_0 T\right],\\
\noalign{\smallskip}\displaystyle	
0 & t\in \left[\gamma_0 T,\gamma_1 T\right],\\
 \noalign{\smallskip}\displaystyle
-\dfrac{1}{(1 - \gamma_T)^2}\widehat{p}\left(\dfrac{T - t}{1 - \gamma_T}\right) & t\in [\gamma_1 T,T],
\end{array}
\right.
\end{equation*}
      \begin{equation*}
      V_l(t) :=
      \left\{
      \begin{array}{ll}
       \gamma_0^{-1}\, \widetilde v_l\left(t\,\gamma_0^{-1},x\right)& t\in \left[0,\gamma_0 T\right],\\
    \noalign{\smallskip}\displaystyle	
      0 & t\in \left[\gamma_0 T,\gamma_1 T\right],\\
      \noalign{\smallskip}\displaystyle
      \dfrac{1}{1 - \gamma_T}\widehat{v}_r\left(\dfrac{T - t}{1 - \gamma_T}\right) & t\in [\gamma_1 T,T]
      \end{array}
      \right.
      \end{equation*}
and
    \begin{equation*}
     V_r(t) :=
      \left\{
      \begin{array}{ll}
      \gamma_0^{-1}\, \widetilde v_r\left(t\,\gamma_0^{-1},x\right) & t\in \left[0,\gamma_0 T\right],\\
     \noalign{\smallskip}\displaystyle	
      0 & t\in \left[\gamma_0 T,\gamma_1 T\right],\\
      \noalign{\smallskip}\displaystyle
      \dfrac{1}{1 - \gamma_T}\widehat{v}_l\left(\dfrac{T - t}{1 - \gamma_T}\right) & t\in [\gamma_1  T,T].
      \end{array}
      \right.
      \end{equation*}
	
	 It is now straightforward to check that $(Y,Z)\in C^1([0,T]\times [0,L];\mathbb{R}^2)$, $P\in C^0([0,T])$, $V_l,V_r \in C^1([0,T])$ and \eqref{es14} and \eqref{mg3} are satisfied.


\section{Global controllability of the viscous Burgers-$\alpha$ system}\label{sec:viscous}


\subsection{Smoothing effect} \label{smooth} 

The goal of this section is to prove that, starting from a $H^1_0$ initial data, there exists a small time where the solution begins to be smooth. 
More precisely, we have the following result:
\begin{prop}\label{re9}
    Let $y_0\in H^1_0(0,L)$ be given and let $(y^\alpha,z^\alpha)$ be the solution to

   \begin{equation}\label{re1}
   \left\{
   \begin{array}{lll}
   y^\alpha_t - y^\alpha_{xx} + z^\alpha y^\alpha_x = 0 						&\mbox{in}&(0,T)\times (0,L),\\
   z^\alpha - \alpha^2 z^\alpha_{xx} = y^\alpha 								&\mbox{in}&(0,T)\times (0,L),\\
   y^\alpha(\cdot,0) = y^\alpha(\cdot,L) = z^\alpha(\cdot,0) = z^\alpha(\cdot,L) = 0 	&\mbox{in}&(0,T),\\
   y^\alpha(0,\cdot) = y_0 												&\mbox{in}&(0,L).
   \end{array}
   \right.
   \end{equation}
    Then, there exist~$T^*\in(0,T/2)$~and~$C > 0$ (independent of $\alpha$) such that the solution 
    $y^\alpha$ belongs to $C^0([T^*, T];C^2([0,L]))$ and satisfies
       \begin{equation*}\label{re13}
       \|y^\alpha\|_{C^0([T^*, T];C^2([0,L]))} \leq \Lambda (\|y_0\|_{H_0^1}),
       \end{equation*}
       where $\Lambda:\mathbb{R}_+ 	\to \mathbb{R}_+$ is
        a continuous function satisfying $\Lambda(s)\to 0$ as $s \to 0^+$.
   \end{prop}

\begin{proof}
	We will divide the proof in several steps. Throughout the proof, all the constants are independent of $\alpha$.
	
	\
       
         	\noindent{\bf Step 1: Strong estimates in $(0,T/2)$.} Since $y_0\in H^1_0(0,L)$, $f \equiv 0$ and $v_l\equiv v_r \equiv 0$, Proposition~\ref{re10} implies the
	existence and uniqueness of a solution $(y^\alpha,z^\alpha)$ to \eqref{re1} satisfying~\eqref{mg6} and \eqref{mg5}. 
	In particular, we have from \eqref{mg5} that
      \begin{equation*}\label{re12}
      \|y^{\alpha}\|_{L^2(H^2\cap H^1_0)} \leq C_1\|y_0\|_{H_0^1}e^{C\|y_0\|^2_{H_0^1}}.
      \end{equation*}      
      
      Consequently, there exists $ t_1\in( 0,T/2)$ such that
	\begin{equation*}\label{est_t_1}
		  \|y^{\alpha}(t_1,\cdot)\|_{H^2\cap H^1_0} \leq \sqrt{\dfrac{2}{T}}C_1\|y_0\|_{H^1_0}e^{C\|y_0\|^2_{H_0^1}}.
	\end{equation*}

\ 

\

         \noindent{\bf Step 2: Estimates in $(t_1,T/2)$.} Let us set $y_1:=y^{\alpha}(t_1,\cdot)$, $g := z^\alpha y^\alpha_x$. Then, we can easily check that $y^\alpha$ 
	is the unique solution to the heat equation:
	         \begin{equation}\label{mg7}
	         \left\{
	         \begin{array}{lll}
	          y^\alpha_t - y^\alpha_{xx} = g 				&\mbox{in}& (t_1,T)\times (0,L),\\
	          y^\alpha(\cdot,0) = y^\alpha(\cdot,L) = 0 		&\mbox{in} & (t_1,T),\\
	          y^\alpha(t_1,\cdot) = y_1 					&\mbox{in} & (0,L).
	         \end{array}
	         \right.
	         \end{equation}
	         From the regularity of $y^\alpha$ and $z^\alpha$, we have $g\in L^2(0, T; H_0^1(0,L))\cap  H^1(0, T; H^{-1}(0,L))$ and	         
	         \begin{align*}
	         \|g\|_{L^2(H_0^1)} + \|g_t\|_{L^2(H^{-1})}
	         &\leq\, 
	         C\|y^{\alpha}\|_{L^{\infty}( H_0^1)}(\|y^{\alpha}\|_{L^2(H^2)} + \|y_t^{\alpha}\|_{L^2(L^2)})\\
	         & \leq \,Ce^{C\|y_0\|_{H_0^1}^2}\|y_0\|_{H_0^1}^2.
	         \end{align*}
	        
	         Using this estimate, the fact that $g\in C^0(0,T;L^2(0,L))$ (see \cite[Ch. $5$, Thm.~3]{evans}), \eqref{mg7} and the parabolic regularity result  \cite[Ch. $3$, Thm. $5$]{evans}, we find that
	         \[
	            y^{\alpha} \in L^{\infty}(t_1, T; H^2(0,L)),\,\, y_t^{\alpha} \in L^{\infty}(t_1, T; L^2(0,L))\cap L^2(t_1, T; H_0^1(0,L))\cap 
	            H^1(t_1, T; H^{-1}(0,L))
	         \]
            and, in the time interval $(t_1,T)$,
            \begin{equation}\label{mg8}
              \begin{alignedat}{2}
              \|y^{\alpha}\|_{L^{\infty}(H^2)} + \|y^{\alpha}_t\|_{L^{\infty}(L^2)\cap L_1^2(H_0^1)} + \|y^{\alpha}_{tt}\|_{L^2(H^{-1})} \leq&~ C\left(\|g\|_{L^2(H_0^1)\cap H^1(H^{-1})} + \|y_1\|_{H^2} \right)\\
              \leq&~ {1\over 2}\Lambda_1(\|y_0\|_{H_0^1})          \end{alignedat}
             \end{equation}
              where 
              $$\Lambda_1 (\|y_0\|_{H_0^1}) = 2Ce^{C\|y_0\|_{H_0^1}^2}\|y_0\|_{H_0^1}(1 + \|y_0\|_{H_0^1}). 
              $$
             From \eqref{mg7}, we have that 
             \begin{equation*}\label{mg19}
             \left\{
             \begin{array}{l}
             -y^\alpha_{xx}(t,\cdot) = g(t,\cdot) - y^\alpha_t(t,\cdot)\\
             y^\alpha(t,0) = y^\alpha(t,L) = 0
             \end{array}
             \right.
             \end{equation*}
             for $t$ a.e in $(t_1,T)$. Thus, using  \eqref{mg8} and elliptic regularity results, (see \cite[Ch. $6$, Thm. $5$]{evans}), we deduce that $y^\alpha\in L^2(t_1,T;H^3(0,L))$ and
             \begin{equation*}\label{mg9}
             \|y^\alpha\|_{L^2(t_1,T;H^3(0,L))} \leq {1\over 2}\Lambda_1 (\|y_0\|_{H_0^1}).
             \end{equation*}  
             
             We also deduce that, for some $ t_2\in( t_1,T/2)$, one has
             	\begin{equation*}\label{mg10}
		  \|y^{\alpha}_t(t_2,\cdot)\|_{H^1_0}+ \|y^{\alpha}(t_2,\cdot)\|_{H^3\cap H^1_0} \leq \sqrt{\dfrac{2}{T-2t_1}}\Lambda_1 (\|y_0\|_{H_0^1}).
	\end{equation*}
             
             \

             \noindent{\bf Step 3: Estimates in $(t_2,T/2)$.} Let us set $y_2:=y^{\alpha}(t_2,\cdot)$. Note that
                \begin{align*}
	         \|g\|_{L^2(t_1,T;H^2(0,L))\cap H^1(t_1,T;L^2(0,L))}
	         &\leq\, 
	         C\|y^{\alpha}\|_{L^{\infty}(t_1,T;H_0^1(0,L))}\|y^{\alpha}\|_{L^2(t_1,T;H^3(0,L))\cap H^1(t_1,T;H_0^1(0,L))}\\
	         & \leq \,C\|y_0\|_{H_0^1}\Lambda_1 (\|y_0\|_{H_0^1})\,e^{C\|y_0\|_{H_0^1}^2}
	         \end{align*}
             and the needed compatibility conditions for regularity results holds:
             $$
             g(t_2,\cdot) + (y_2)_{xx}(t,\cdot) = y^\alpha_t(t_2,\cdot) \in H_0^1(0,L).
             $$
             Using \cite[Ch. $7$, Thm. $6$]{evans}, we get that 
             $$
             y^\alpha\in L^2(t_2,T;H^4(0,L))
             \cap H^1(t_2,T;H^2(0,L))
             \cap H^2(t_2,T;L^2(0,L))
             $$
             and, moreover, in the time interval $(t_2,T)$
             \begin{equation}\label{mg12}
             \begin{alignedat}{2}
             \|y^\alpha\|_{L^2(H^4)
             \cap H^1(H^2)
             \cap H^2(L^2)}
              \leq&~ C \left(\|g\|_{L^2(H^2)\cap H^1(L^2)} + \|y_2\|_{H^3}\right)\\
             \leq & \Lambda_2 (\|y_0\|_{H_0^1}),
             \end{alignedat}
             \end{equation}
        where
        $$
        \Lambda_2 (\|y_0\|_{H_0^1}):= C\left(1+\|y_0\|_{H_0^1}e^{C\|y_0\|_{H_0^1}^2}\right)\Lambda_1 (\|y_0\|_{H_0^1}).
        $$

	
               \

              \

             \noindent{\bf Step 4: Conclusion.}  Finally, the result in \cite[Ch. 5, Thm. 4]{evans} applied to \eqref{mg12} leads to the regularity $C^0([t_2,T]; H^3(0,L))$ for  $y^\alpha$. Therefore, the conclusion follows from Sobolev's embedding, taking $T^* = t_2$ and $\Lambda (\|y_0\|_{H_0^1})=\Lambda_2 (\|y_0\|_{H_0^1})$. 
             
             
           
  	       \end{proof} 
	
		\begin{remark}{\rm Proposition \ref{re9} is also true when $y_0\in L^\infty(0,L)$. Indeed, we can start using Proposition \ref{re10} that guarantees the
	existence and uniqueness of a solution $(y^\alpha,z^\alpha)$ to \eqref{re1} satisfying~\eqref{mg66} and \eqref{mg55}. 
	In particular, we have from \eqref{mg55} that
      \begin{equation*}\label{re12222}
      \|y^{\alpha}\|_{L^2(H^1_0)} \leq C_1\|y_0\|_{\infty}e^{C\|y_0\|^2_{\infty}}.
      \end{equation*}      
      Therefore, there exists $ t_1\in( 0,T/2)$ such that
	\begin{equation*}\label{est_t_111}
		  \|y^{\alpha}(t_1,\cdot)\|_{H^1_0} \leq \sqrt{\dfrac{2}{T}}C_1\|y_0\|_{\infty}e^{C\|y_0\|^2_{\infty}}.
	\end{equation*}
    Then, we can achieve arguing as in the proof of Proposition \ref{re9}.
}
	\end{remark}
    

\subsection{Uniform approximate controllability}\label{approximate}

	In this section, the goal is to prove the following approximate controllability result starting from sufficiently smooth initial data:
  \begin{prop}\label{re}
  Let $y_0, y_f\in C^2([0,L])$ be given. There exist
  positive constants $\tau_*$ and $K>0$, independent of $\alpha$,
  such that, for any $\tau\in (0,\tau_*]$,  
  there exist $p^{\alpha}\in C^0([0,\tau])$, 
  $(v^{\alpha}_l,v^{\alpha}_r)\in H^{3/4}(0,\tau)\times H^{3/4}(0,\tau)$ and associated states $(y^\alpha,z^\alpha)$ with the following regularity
    \begin{equation} \label{sp25}
     \left\{
     \begin{array}{l}
     y^\alpha \in L^2(0,\tau; H^2(0,L))\cap H^1(0,\tau;L^2(0,L))\cap C^0([0,\tau];H^1(0,L))\\
     z^\alpha\in L^2(0,T; H^4(0,L))\cap H^1(0,\tau;L^2(0,L))\cap C^0([0,\tau]; H^3(0,L)),    
     \end{array}
     \right.
     \end{equation}
   satisfying
  
     \begin{equation}
     \left\{
     \begin{array}{lll}
     y^\alpha_t - y^\alpha_{xx} + z^\alpha y^\alpha_x = p^{\alpha}(t) &\mbox{in}& (0,\tau)\times (0,L),\\
     z^\alpha - \alpha^2 z^\alpha_{xx} = y^\alpha &\mbox{in}& (0,\tau)\times (0,L),\\  \label{sp26}
     z^\alpha(\cdot,0) = y^\alpha(\cdot,0) = v^{\alpha}_l &\mbox{on}&(0,\tau),\\ 
     z^\alpha(\cdot,L) =  y^\alpha(\cdot,L) = v^{\alpha}_r  &\mbox{on}& (0,\tau),\\
     y^\alpha(0,\cdot) = y_0 &\mbox{in}&(0,L)\\
     \end{array}
     \right.
     \end{equation}
     and, moreover,
     \begin{equation}\label{sp27}
     \|y^\alpha(\tau,.)-y_f\|_{H^1(0,L)}\leq K\sqrt{\tau}
     \end{equation}
     and
\begin{equation*}\label{sp277}
    \|p^{\alpha}\|_{C^0([0,T])}+
	\|(v_l^{\alpha},v_r^{\alpha})\|_{H^{3/4}([0,T];\mathbb{R}^2)}\leq C \quad \forall \alpha>0.
\end{equation*}
    \end{prop}

	In order to prove this result, let us introduce $\lambda\in C_0^1(0,1)$ with
	$\|\lambda\|_{L^1(0,1/2)} > L$ and~$\lambda(t)=\lambda(1-t)$ for all $t\in [0,1]$.
	Let us set $\lambda^{\tau}(t) := \frac{1}{\tau}\lambda\left(\frac{t}\tau\right)$ for all $t \in [0,\tau]$.
   
    The following two results hold:        
     \begin{lem} \label{re8}
    Let $M > 0$ be a positive constant. Then, if $u_0, u_f\in C^2([0,L])$ and
      \begin{align}\label{apr4}
     \max\{\|u_0\|_{C^2([0,L])},\|u_f\|_{C^2([0,L])}\}\leq M,  
      \end{align}
    there exists $\tau_0\in (0,1)$ such that for every 
    $\tau\in (0,\tau_0]$ we can find controls $v_l^{\alpha,\tau},v_r^{\alpha,\tau}$ in $ C^2([0,\tau])$ 
    and associated states $u^{\alpha,\tau}, w^{\alpha,\tau}$ in $C^2([0,\tau]\times ([0,L]))$, satisfying
     \begin{equation}
     \left\{
     \begin{array}{lll}
     u^{\alpha,\tau}_t + (\lambda^{\tau}(t) + w^{\alpha,\tau})u^{\alpha,\tau}_{x}  = 0 &$\mbox{in}$& (0,\tau)\times (0,L),\\
     w^{\alpha,\tau} - \alpha^2 w^{\alpha,\tau}_{xx} = u^{\alpha,\tau} &$\mbox{in}$& (0,\tau)\times (0,L),\\  
     u^{\alpha,\tau}(\cdot,0) = w^{\alpha,\tau}(\cdot,0) = v_l^{\alpha,\tau} &$\mbox{in}$& (0,\tau),\\
     u^{\alpha,\tau}(\cdot,L) = w^{\alpha,\tau}(\cdot,L) = v_r^{\alpha,\tau} &$\mbox{in}$& (0,\tau),\\   \label{re6}
     u^{\alpha,\tau}(0,\cdot) = u_0 &$\mbox{in}$&(0,L),\\
     u^{\alpha,\tau}(\tau,\cdot) = u_f &$\mbox{in}$&  (0,L).
     \end{array}
      \right.
     \end{equation}  
    Furthermore, there exists $C > 0$, independent of $\alpha$ and $\tau$, such that
      \begin{equation}\label{re7}
        \|u^{\alpha,\tau}\|_{C^0([0,\tau];C^2([0,L]))}\leq CM. 
      \end{equation}
\end{lem}
\begin{proof}
	First, thanks to the fact that $\|\lambda\|_{L^1(0,1/2)} > L$ and Remark \ref{rm:1}, we know that  there exists $\delta > 0$ (independent of $\alpha$) such that, for any initial datum in a ball of  $C^2([0,L])$ centered at origin and radius $\delta$, there exists a solution 
	to \eqref{sp1} belonging to  
	$C^2([0,1/2]\times [0,L])$ satisfying
	\eqref{eq:c2}.

	   Let us now take $\tau_0 \in (0,1)$ such that $\tau_0 M \leq \delta$. 
	      Then, according to the previous construction, for each $\tau\in (0,\tau_0]$ there exist functions $(\widetilde{y}^\alpha,\widetilde{z}^\alpha)$, $(\widehat{y}^\alpha,\widehat{z}^\alpha) \in C^2([0,1/2]\times [0,L];\mathbb{R}^2))$, solutions to \eqref{sp1} and satisfying
 $\widetilde{y}^\alpha(0,x)=\tau u_0(x)$ and
	      $\widehat{y}^\alpha(0,x) = \tau u_f(L - x)$, for all $x\in [0,L]$, 
	      and \eqref{eq:c2}. 

	   Then, one defines the states 
	       \begin{equation*}
	       u^{\alpha,\tau}(t,x) :=
	        \left\{
	        \begin{array}{lll}
	        \tau^{-1}\widetilde{y}^\alpha(\tau^{-1}t,x) &\mbox{in}& [0,\tau/2]\times [0,L],\\
	        \tau^{-1}\widehat{y}^\alpha(\tau^{-1}(\tau - t),L-x) &\mbox{in} & [\tau/2,\tau]\times [0,L]
	        \end{array}
	        \right.
	       \end{equation*}
	        and
	       \begin{equation*}
	       w^{\alpha,\tau}(t,x) :=
	        \left\{
	        \begin{array}{lll}
	        \tau^{-1}\widetilde{z}^\alpha(\tau^{-1}t,x) &\mbox{in}& [0,\tau/2]\times [0,L],\\
	       \tau^{-1}\widehat{z}^\alpha(\tau^{-1}(\tau - t),L-x) &\mbox{in} & [\tau/2,\tau]\times [0,L],
	        \end{array}
	        \right.
	       \end{equation*}
	        that satisfy $(u^{\alpha,\tau},w^{\alpha,\tau})\in C^2([0,\tau]\times[0,L];\mathbb{R}^2)$
	       and the associated boundary controls 
	       \[
	       v^{\alpha,\tau}_l(t) := u^{\alpha,\tau}(t,0)
	       \quad
	      \text{and}
	       \quad
	       v^{\alpha,\tau}_r(t) :=u^{\alpha,\tau}(t,L).
	       \]
	   Since 
	   $\lambda(\tau^{-1}t) \equiv \lambda(\tau^{-1}(\tau-t))$, 
	   the couple $(u^{\alpha,\tau},w^{\alpha,\tau})$
	   satisfies~\eqref{re6} and \eqref{re7}.
    \end{proof} 
  
     \begin{lem} 
     Assume that $M > 0$, $u_0, u_f\in C^2([0,L])$ satisfy \eqref{apr4} and $\tau_0$ is furnished by Lemma  \ref{re8}. There exists $\tau_*\in (0,\tau_0]$ such that, for any $\tau\in (0,\tau_*]$ and any $(u^{\alpha,\tau},w^{\alpha,\tau})\in C^2([0,\tau]\times [0,L];\mathbb{R}^2)$ satisfying \eqref{re6} and \eqref{re7}, 
     there exists a unique solution to
        \begin{equation*} \label{es8}
        \left\{
        \begin{array}{lll}
        r^{\alpha,\tau}_t + (q^{\alpha,\tau} + w^{\alpha,\tau} + \lambda^{\tau})r^{\alpha,\tau}_x - r^{\alpha,\tau}_{xx} + q^{\alpha,\tau} u^{\alpha,\tau}_x - u^{\alpha,\tau}_{xx} = 0&\mbox{in}&(0,\tau)\times(0,L),\\
        q^{\alpha,\tau} - \alpha^2 q^{\alpha,\tau}_{xx} = r^{\alpha,\tau} &\mbox{in}& (0,\tau)\times(0,L),\\ 
        r^{\alpha,\tau}(\cdot,0) =0,\quad  r_x^{\alpha,\tau}(\cdot,L) = 0 &\mbox{in}&(0,\tau),\\
        q^{\alpha,\tau}(\cdot,0) =  0,\quad
      q^{\alpha,\tau}(\cdot,L) = r^{\alpha,\tau}(\cdot,L)&\mbox{in}&(0,\tau),\\
        r^{\alpha,\tau}(0,\cdot) = 0 &\mbox{in}&(0,L),
        \end{array}
        \right.
       \end{equation*}
     satisfying
     \begin{equation*} \label{es9}
     \left\{
     \begin{array}{l}
     r^{\alpha,\tau} \in L^2(0,\tau; H^2(0,L))\cap H^1(0,\tau; L^2(0,L))\cap C^0([0,\tau];H^1(0,L)),\\
     q^{\alpha,\tau}\in L^2(0,\tau; H^4(0,L))\cap H^1(0,\tau; L^2(0,L))\cap C^0([0,\tau]; H^3(0,L))
     \end{array}
     \right.
     \end{equation*}
     and
     \begin{equation*}\label{es111}
   \|r^{\alpha,\tau}\|_{L^2(0,\tau;H^2(0,L))\cap H^1(0,\tau;L^2(0,L))} \leq C.
      \end{equation*}
   Here, $C$ is a positive constant that depends on $L,T,M$ and $\tau$, but it is independent of $\alpha$. 
    Moreover, there exists a constant $K$ that depends on $L,T$ and $M$ (independent of $\alpha$ and $\tau$), such that
    \begin{equation}\label{es11}
   \|r^{\alpha,\tau}\|_{C^0([0,\tau];H^1(0,L))} \leq K\sqrt{\tau}.
    \end{equation}
\end{lem}
\begin{proof}
  The proof is standard. It can be easily obtained, for instance, via a Faedo-Galerkin technique in combination with well known energy estimates.  
\end{proof} 

 We can now achieve the proof of Proposition \ref{re}. Indeed,
given $\tau\in (0,\tau_*]$, it is not difficult to see that $(y^{\alpha},z^{\alpha})$ given by $$(y^{\alpha},z^{\alpha}) := (u^{\alpha,\tau} + r^{\alpha,\tau} + \lambda^{\tau},w^{\alpha,\tau} + q^{\alpha,\tau} + \lambda^{\tau})$$
satisfies \eqref{sp25} and \eqref{sp26} with $p^\alpha(t) = \lambda^{\tau}_t$ and boundary controls $v_l^\alpha(t) = u^{\alpha,\tau}(t,0) + r^{\alpha,\tau}(t,0) + \lambda^{\tau}(t)$
and $v_r^\alpha(t) = u^{\alpha,\tau}(t,L) + r^{\alpha,\tau}(t,L) + \lambda^{\tau}(t)$. Moreover, using $\eqref{re6}_6$, \eqref{es11} and the fact that $\lambda^{\tau}$ vanishes in the neighbourhood of $\tau$, we obtain $\eqref{sp27}$.


\subsection{Uniform local exact controllability to the trajectories}\label{localexact}

   The goal of this section is to prove the local exact controllability to space-independent trajectories for the Burgers-$\alpha$ system, with controls and associated states uniformly bounded with respect to $\alpha$ in appropriate spaces. Thus, let $\widehat{m}\in C^1([0,T])$ be given and note that $(\widehat{y}^{\alpha},\widehat{z}^{\alpha}) = (\widehat{m},\widehat{m})$ is a trajectory of viscous Burgers-$\alpha$ system with $(\widehat{p}^{\alpha}(t),\widehat{v}_l^{\alpha}(t),\widehat{v}_r^{\alpha}(t)) = (\widehat{m}'(t),\widehat{m}(t),\widehat{m}(t))$. We have the following result:
   \begin{thm}\label{consttrajcont}
   Let $T, L, \alpha > 0$ and $\widehat{m}\in C^1([0,T])$ be given. There exists $\delta >0$ (independent of $\alpha$) such that, for any initial data $y_0\in H^1(0,L)$ satisfying $\|y_0 - \widehat{m}(0)\|_{H^1}\leq \delta$ there exist  $p^{\alpha} \in C^0([0,T])$ and $(v_l^{\alpha}, v_r^{\alpha}) \in H^{3/4}(0,T;\mathbb{R}^2)$ and associated states $(y^{\alpha},z^{\alpha}) \in L^2(0,T;H^2(0,L;\mathbb{R}^2))\cap H^1(0,T;L^2(0,L;\mathbb{R}^2))$ satisfying \eqref{viscous} and 
      \begin{align}\label{re27}
        y^{\alpha}(T,\cdot)\equiv z^{\alpha}(T,\cdot)\equiv \widehat{m}(T). 
      \end{align}
          Moreover, $p^{\alpha}=\widehat{m}'$ and the following estimates hold:
\begin{equation}\label{v-unif_local}
    \|p^{\alpha}\|_{C^0([0,T])}+
	\|(v_l^{\alpha},v_r^{\alpha})\|_{H^{3/4}([0,T];\mathbb{R}^2)}\leq C \quad \forall \alpha>0,
\end{equation}
    where $C>0$ is a positive constant independent of $\alpha$.
   \end{thm}
   
    Let us set $(y^{\alpha}, z^{\alpha}) = (u^{\alpha} + \widehat{m},w^{\alpha} + \widehat{m})$ 
    and $p^{\alpha}=\widehat{m}'$. Then, $(u^{\alpha},w^{\alpha})$ must satisfy
         \begin{equation}\label{consttr}
    \left\{
    \begin{array}{lll}
     u^{\alpha}_t - u^{\alpha}_{xx} + (w^{\alpha} + \widehat{m})u^{\alpha}_x =  0  &\mbox{in} & (0,T)\times (0,L),\\
    w^{\alpha} - \alpha^2 w^{\alpha}_{xx} = u^{\alpha} &\mbox{in} & (0,T)\times (0,L),\\
    u^{\alpha}(\cdot,0) = w^{\alpha}(\cdot,0) = h^{\alpha}_l&\mbox{in} & (0,T),\\
    u^{\alpha}(\cdot,0) = w^{\alpha}(\cdot,L) = h^{\alpha}_r &\mbox{in} & (0,T),\\
    u^{\alpha}(0,\cdot) = u_0 &\mbox{in} & (0,L),
    \end{array}
    \right.
    \end{equation}
 where $u_0 := y_0 - \widehat{m}(0)$ and $(h^{\alpha}_l,h^{\alpha}_r):=(v^{\alpha}_l-\widehat{m},v^{\alpha}_r-\widehat{m})$. Therefore, Theorem \ref{consttrajcont} is equivalent to the local null-controllability to \eqref{consttr}. 
     \begin{prop*}
              Let the conditions of Theorem \ref{consttrajcont} be satisfied. There exists $\delta >0$ (independent of $\alpha$) such that, for any initial data $u_0\in H^1(0,L)$ satisfying $\|u_0\|_{H^1}\leq \delta$, there exist $(h_l^{\alpha}, h_r^{\alpha}) \in H^{3/4}(0,T;\mathbb{R}^2)$ and  $(u^{\alpha},w^{\alpha}) \in L^2(0,T;H^2(0,L;\mathbb{R}^2))\cap H^1(0,T;L^2(0,L;\mathbb{R}^2))$ satisfying \eqref{consttr} and 
      \begin{align}\label{consttr1}
        u^{\alpha}(T,\cdot)\equiv w^{\alpha}(T,\cdot)\equiv0. 
      \end{align}
          Moreover, there exists a positive constant $C>0$ (independent of $\alpha$) such that
\begin{equation}\label{v-unif_constt}
	\|(h_l^{\alpha},h_r^{\alpha})\|_{H^{3/4}([0,T];\mathbb{R}^2)}\leq C \quad \forall \alpha>0.
\end{equation}
     \end{prop*}
\begin{proof} The proof of this result relies on a fixed-point argument. 
Thus, given $u_0\in H^1(0,L)$ and $\eta > 0$, one can get by reflection method an extension $u_0^*\in H_0^1(-\eta, L + \eta)$, with 
\begin{align*}\label{re15}
    \|u^{*}_0\|_{H_0^1(-\eta, L+\eta)} \leq C \|u_0\|_{H^1(0,L)}.
\end{align*}
         Let $R > 0$ be given and consider the set
          \begin{align*}
         B_{R}^{\eta}:= \{\bar{u}\in  L^{\infty}(0,T;C^0([-\eta, L + \eta]): \,\, \|\bar{u}\|_{L^{\infty}(0,T;C^0([-\eta,L+\eta])) }\leq R\}.
          \end{align*}
            
         For any $\bar{u}\in B_{R}^{\eta}$, we can easily deduce that there exists a unique solution to
         \begin{equation}\label{re16}
         \left\{
         \begin{array}{lll}
         w - \alpha^2 w_{xx} = \bar{u}1_{(0,L)} &\mbox{in}& (0,T)\times (0,L),\\
         w(\cdot, 0) = \bar{u}(\cdot, 0),\,\ w(\cdot, L) = \bar{u}(\cdot, L) &\mbox{in}& (0,T).
         \end{array}
         \right.
         \end{equation}
         Moreover, using the maximum principle, we obtain that
         \begin{align*}
         \|w\|_{L^{\infty}(0,T;C^0([0,L]))} &\leq C\|\bar{u}\|_{L^{\infty}(0,T;C^{0}([-\eta,L+\eta]))}\leq CR.
         \end{align*}
         Then, again reflection method, we get an extension $ w^*\in L^{\infty}(0,T; C^2([-\eta, L + \eta]))$ with
         \begin{align*}\label{re17}
         \|w^*\|_{L^{\infty}(0,T;C^0([-\eta,L+\eta]))}&
         \leq C\|w\|_{L^{\infty}(0,T;C^0([0,L]))} \leq CR.
         \end{align*}
         
          We assume that $ L < a < b < L + \eta$. Then, arguing as in the proof of \cite[Theorem $1$]{Araruna}, we find $v\in L^\infty((0,T)\times (a,b))$ and $u\in L^2(0,T;H^2(-\eta,L+\eta))\cap L^\infty(0,T;H^1_0(-\eta,L+\eta))$ such that
         \begin{equation}\label{re24}
         \left\{
         \begin{array}{lll}
         u_t - u_{xx} + (w^* + \widehat{m})u_x = v1_{(a,b)}&\mbox{in}& (0,T)\times (-\eta,L + \eta),\\
         u(\cdot,-\eta) = u(\cdot,L + \eta) = 0 &\mbox{in} & (0,T),\\
        u(0,\cdot) = u_0^{*} &\mbox{in}& (-\eta,L + \eta),\\
        u(T,\cdot) = 0 &\mbox{in}& (-\eta,L + \eta), 
         \end{array}
         \right.
         \end{equation}
    and
\begin{equation*}\label{re28}
    \|v\|_{L^\infty(0,T;L^\infty(a,b))} \leq C\|u_0\|_{H^1(0,L)},
\end{equation*}
    for some $C > 0$ of the form
\[
    C:=e^{C_0
    [1 + 1/T + (1 + T)(\|w^*\|_{L^\infty(L^\infty)}^2
    +\|\widehat{m}^2_{\infty})]}.
\]
where $C_0 > 0$ depends on $a,b, L$ and $\eta$. Therefore, it is not difficult to deduce that the norm of $u$ in $H^1(0,T;L^2(-\eta,L+\eta))$, $L^2(0,T;H^2(-\eta,L+\eta))$ and $L^\infty(0,T;H^1_0(-\eta,L+\eta))$ are bounded by $C\|u_0\|_{H^1}$, where $C$ is independent of $\alpha$.     

         Consequently, there exists $\delta > 0$ (independent of $\alpha$) such that, if $\|u_0\|_{H^1}\leq \delta$, one has $\|u\|_{L^{\infty}(0,T;C^0([-\eta,L+\eta]))}\leq R$ and the mapping $ \Lambda_{\alpha} : B^{\eta}_{R} \mapsto B^{\eta}_{R}$,  $\Lambda_{\alpha}(\bar{u}) := u$ is well defined. Note that
          \begin{enumerate}
          \item $\Lambda_{\alpha}$ is well defined and continuous. Indeed, this follows from the uniqueness of solution of \eqref{re16} and \eqref{re24}; the continuity is obtained by using standard parabolic estimates and the fact that, if $\bar{u}_n \rightarrow \bar{u}$ in $L^{\infty}(0,T; C^0([-\eta, L + \eta]))$, then $w^{*}_{n} \rightarrow w^{*}$ in $L^{\infty}(0,T; C^0([-\eta, L + \eta]))$ and, therefore,
          $u_{n} \rightarrow u$ in $L^{\infty}(0,T; C^0([-\eta, L + \eta]))$.
          
          \item $F^{\eta} := \Lambda_{\alpha}(B^{\eta}_{R})$ is relatively compact in $L^{\infty}(0,T; C^0([-\eta, L + \eta]))$. Indeed, one easily obtains that $F^{\eta}$ is bounded in $L^{\infty}(0,T; H_{\eta,0}^1(-\eta, L + \eta))$ and $F_t^{\eta}$ is bounded in $L^2(0,T; L^2(-\eta, L + \eta))$. Hence, applying again \cite[Corollary $4$]{simon}, we get the
          desired compactness.
           \end{enumerate}
           Finally, by applying Schauder's Fixed-Point Theorem, we see that there exists $u\in B_R^{\eta}$ such that $\Lambda_{\alpha}(u) = u$. Then, the couple $(u^{\alpha},v^{\alpha})$, where $u^{\alpha}$ is the restriction to $(0,T)\times (0,L)$ of $u$ and $w$ is the solution to \eqref{re16}, belongs to $L^2(0,T;H^2(0,L;\mathbb{R}^2))\cap H^1(0,T;L^2(0,L;\mathbb{R}^2))$ and satisfies
          \eqref{consttr}, \eqref{consttr1} and \eqref{v-unif_constt}
          with controls $h^{\alpha}_l := u(\cdot,0)$ and $h^{\alpha}_r := u(\cdot,L)$.
\end{proof}


\subsection{Global exact controllability}

	   In this section we prove Theorem \ref{thm2} by combining the results obtained in Sections \ref{smooth}, \ref{approximate} and \ref{localexact}. First recall that given $y_0\in L^{\infty}(0,L)$ and the unique associated solution $(y^{\alpha}_1,z^{\alpha}_1)$ to \eqref{re1}, Proposition \ref{re9} provides a time
	   $T^* \in (0,T/2)$ and a constant $M^* > 0$ (both independent of $\alpha$) such that $y^{\alpha}_1 \in C^0([T^*,T];C^2([0,L]))$ and, moreover,
	         \begin{align}\label{smooth1}
	             \|y^{\alpha}_1\|_{C^0([T^*,T];C^2([0,L]))}\leq M^*.
	         \end{align}
	         
	   Now, let us fix $N\in\mathbb{R}$, let us set 
	   $M := \max\{M^*,|N|\}$ and assume that the constant $\tau^* >0$, furnished by Proposition \ref{re} is small enough, such that $T^* < T/2 - \tau^*$. Then, $y^{\alpha}_{2,0}:= y^{\alpha}_1(T/2 - \tau,\cdot)$ belongs to $C^2([0,L])$ and, from \eqref{smooth1} and Proposition \ref{re}, there exist $p^{\alpha}_2\in C^0([0,\tau])$, $(v^{\alpha}_{l,2},v^{\alpha,2}_{r,2}) \in H^{3/4}(0,\tau;\mathbb{R}^2)$ and associated states $(y^{\alpha}_2,z^{\alpha}_2)\in L^2(0,T;H^2(0,L;\mathbb{R}^2))\cap H^1(0,T;L^2(0,L;\mathbb{R}^2))$ satisfying \eqref{sp25}, \eqref{sp26} and \eqref{sp27}, with initial datum $y^{\alpha}_{2,0}$ and target $y_f = N$. 
	         
	         Finally, decreasing $\tau$ if necessary and setting $y^{\alpha}_{3,0} := y^{\alpha}_2(\tau,\cdot)$, we deduce, thanks to \eqref{sp27}, that $\|y^{\alpha}_{3,0} - N\|_H^1 \leq \delta$, where $\delta > 0$ is the constant given in Theorem \ref{consttrajcont} for a control time $T/2$. Hence, this theorem (applied with $\widehat{m}\equiv N$), guarantees the existence of controls $(v^{\alpha}_{l,3},v^{\alpha}_{r,3})\in H^{3/4}(0,T/2;\mathbb{R}^2)$ such that the associated states $(y^{\alpha}_3,z^{\alpha}_3)$
	         satisfying \eqref{viscous}, \eqref{re27} and \eqref{v-unif_local}, with $p^\alpha\equiv 0$ and initial datum $y^{\alpha}_{3,0}$.

	         To conclude, using $(y^{\alpha}_1,z^{\alpha}_1)$, $(y^{\alpha}_2,z^{\alpha}_2)$ and $(y^{\alpha}_3,z^{\alpha}_3)$, and the associated controls, we can build the required solution, as stated in Theorem \ref{thm2}.

\section{Additional comments and questions}\label{sec:comments}

\subsection{Passage to the limit when $\alpha \rightarrow 0$}

    Theorem \ref{thm1} establishes the existence of uniformly bounded controls for the inviscid Burgers-$\alpha$ equation; the family of associated solutions is uniformly bounded in $C^1([0,T]\times [0,L])$. What happens as $\alpha$ goes to $0$? For uncontrolled nonlocal conservation law, a similar question related to singular limit was studied in \cite{colombo}.
 
    Thanks to Theorem \ref{thm2}, assuming that $y_0\in H_0^1(0,L)$, 
    the family of controls $\{(p^{\alpha},v^{\alpha}_l,v^{\alpha}_r)\}_{\alpha > 0}$ of the viscous Burgers-$\alpha$ systems is uniformly bounded in $C^0([0,T])\times H^{3/4}(0,T;\mathbb{R}^2)$ and the associated family of states $\{y^{\alpha}\}_{\alpha > 0}$ is uniformly bounded in $L^2(0,T;H^2(0,L))\cap H^1(0,T;L^2(0,L))$. It is not difficult to verify that $\{y^{\alpha}\}_{\alpha > 0}$ converges, as $\alpha$ goes to $0$, to a controlled solution to the Burgers equation with same initial datum $y_0$.
    
    An additional interesting question is to determine the order of convergence of $y^\alpha$, in the convergence space.
 
 \subsection{Null controllability with $2$ controls}
 
      In Theorems \ref{thm1} and \ref{thm2}, we have used $3$ scalar controls. It remains open to see whether, using arguments similar to those in \cite{marbach}, it is also possible to prove global uniform null controllability with only $2$ scalar controls. 
 
 \subsection{Global exact controllability to the trajectories}
 
    
       At least two additional questions remain open here: (i) to obtain uniform global exact controllability to trajectories for the viscous Burgers-$\alpha$ with trajectories in 
       $W^{1,\infty}(0,T;W^{1,\infty}(0,L;\mathbb{R}^2))$; (ii) to reduce the number of scalar controls.
 
 
 
 \subsection{Less regular initial conditions}
 
 In \cite{marbach}, the author proved a null controllability result for the Burgers equation with initial datum in $L^2(0,L)$. Is it also possible to control uniformly $L^2$ initial conditions in the case of the Burgers-$\alpha$ system? 
 
 
 
 


\begin{thebibliography}{90}

\bibitem{Araruna}
	{\sc F.D. Araruna, E. Fern\'andez-Cara and Diego A. Souza}, {\it On the control of the Burgers-alpha model}, Adv. Differential Equations, {$\bf 18$} ($9$-$10$), $935$-$954$, ($2013$).
	
	\bibitem{lerayalpha}
	{\sc F.D. Araruna, E. Fern\'andez-Cara and Diego A. Souza}, {\it Uniform local null control of the Leray-$\alpha$ model}, ESAIM Control Optim. Calc. Var.,
	{$\bf 20$} ($4$), $1181$-$1202$, ($2014$). 
	
\bibitem{bardos} 
	{\sc C. Bardos and U. Frisch}, {\it Finite-time regularity for bounded and unbounded ideal
              incompressible fluids using {H}\"{o}lder estimates}, Proceedings of the conference held at the university of Paris-Sud Orsay, France (1975), Springer-Verlag, Lectures Notes in Math., {\bf $\bf 565$} ($1976$), no. $9$-$10$, $1$-$13$.
 
 \bibitem{bhat} 
	{\sc H.S. Bhat and R.C. Fetecau}, {\it A Hamiltonian Regularization of the Burgers Equation}, J. Nonlinear Sci., {$\bf 16$} ($6$), $615$-$638$, ($2006$).
                
\bibitem{Chapouly} 
	{\sc M. Chapouly}, {\it Global controllability of nonviscous and viscous Burgers-type equations}, SIAM J. Control Optim., {\bf $\bf 48$} ($2009$), no. $3$, $1567$-$1599$.
	
	
\bibitem {Holm} 
{\sc	A. Cheskidov, D. Holm, E. Olson and E. Titi}, {\it On a Leray-$\alpha$ model of turbulence}, Proc. R. Soc. A, {\bf $\bf 461$} ($2005$), $629$-$649$.


\bibitem {colombo} 
{\sc	M. Colombo, G. Crippa and L. V.Spinolo}, {\it On the singular local limit for conservation laws with
              nonlocal fluxes}, Arch. Ration. Mech. Anal., {\bf $\bf 233$} ($2019$), $1131$-$1167$.


\bibitem{Coron}
{\sc	J.-M. Coron}, {\it Global asymptotic stabilization for controllable systems without drift}, Math. Control Signals Systems, {\bf $\bf5$} ($1992$), pp. $295$–$312$.



\bibitem{Coron1} 
{\sc	J.-M. Coron}, {\it On the controllability of 2-D incompressible perfect fluids}, J. Math. Pures Appl. (9), {\bf $\bf75$} ($1996$), pp. $155$–$188$.

\bibitem{xiang}
 {\sc J. M. Coron and S. Xiang}, {\it Small-time global stabilization of the viscous Burgers equation with three scalar controls}, ($2018$), hal-01723188.



\bibitem{Dias}
{\sc	J. I. D\'iaz}, {\it Obstruction and some approximate controllability results for the Burgers equation and related problems}, Control of Partial Differential Equations and Applications,  Lecture Notes in Pure and Appl. Math., {\bf $\bf 174$} ($1995$), Dekker, New York, $63$-$76$.
	

\bibitem{ode}
 {\sc C.I. Doering and A. O. Lopes}, {\it Equa\c c\~oes diferenciais ordin\'arias}, 5 ed., Cole\c c\~ao matem\'atica universit\'aria, ($2014$), IMPA.
 
 \bibitem{doubova}
 {\sc A. Doubova, E. Fernández-Cara, M. González-Burgos and E. Zuazua}, {\it On
the controllability of parabolic systems with a nonlinear term involving the state and the
gradient }, SIAM J. Control Optim., {$\bf 41$} ($2002$) , n.$3$ , $798–819$.
 
 
\bibitem{evans} 
	{\sc L. Evans}, {\it Partial differential equations},  2 ed.,   Graduate Studies in Mathematics, American Mathematical Society, Providence, {$\bf 19$} ($2010$).
	
\bibitem{fabre}
{\sc  C. Fabre, J.-P. Puel and  E. Zuazua,} {\it Approximate controllability of the semilinear
heat equation}, Proc. Roy. Soc. Edinburgh Sect. A 125, {$\bf 1$} ($1995$), $31$–$61$.	
	

\bibitem{Cara}
{\sc	E. Fern\'andez-Cara and S. Guerrero}, {\it Null controllability of the Burgers system with distributed controls}, Systems \& Control Lett., {\bf $\bf 56$} ($2007$), $366$-$372$.

\bibitem{diego}
{\sc E. Fern\'andez-Cara and Diego A. Souza}, {\it Remarks on the control of a family of $b$–equations}, Trends in Control Theory and Partial Differential Equations, {$\bf 32$} ($2019$), $123-138$.

\bibitem{zuazua}
{\sc E. Fernández-Cara and E. Zuazua,} {\it Null and approximate controllability for weakly blowing up semilinear heat equations}, Ann. Inst. H. Poincaré Anal. Non Linéaire $17$, {$\bf 5$} ($2000$), $583–616$.

\bibitem{Foias} 
{\sc	C. Foias, D. Holm and E. Titi}, {\it The three dimensional viscous Camassa-Holm equation and their relation to the Navier-Stokes equation and turbulence theory}, J. Dynam. Differential Equations, {$\bf 14$} ($2002$), $1$-$36$.

\bibitem{manley}
{\sc C. Foias, C. Manley, O. Rosa, R. and R. Temam}, {\it   Navier-Stokes equations and
turbulence}, Encyclopedia of Mathematics and its Applications, Cambridge University Press, {$\bf 83$} ($2001$). 


\bibitem{Fursikov}
{\sc	A. V. Fursikov and O. Y. Imanuvilov}, {\it Controllability of Evolution Equations}, Lecture Notes,  {\bf $\bf 34$} ($1996$), Seoul National University, Korea .


\bibitem{Glass} 
{\sc	O. Glass}, {\it Exact boundary controllability of 3-D Euler equation}, ESAIM Control Optim. Calc. Var., {\bf $\bf 5$} ($2000$), $1$–$44$.

\bibitem{Glass1}
{\sc	O. Glass and S. Guerrero}, {\it On the uniform controllability of the Burgers equation}, SIAM J. Control Optim., {\bf $\bf46$} ($2007$), $1211$–$1238$.


\bibitem{Guerrero}
{\sc	S. Guerrero and O.Y. Imanuvilov}, {\it Remarks on global controllability for the Burgers equation with two control forces}, Ann. Inst. H. Poincar\'e Anal. Non Lin\'eaire, {\bf $\bf 24$} ($2007$), $897$-$906$.



\bibitem{Horsin}
{\sc	T. Horsin}, {\it On the controllability of the Burgers equation}, ESAIM Control Optim. Calc. Var., { $\bf3$} ($1998$), $83$-$95$.

\bibitem{marbach}
{\sc F. Marbach}, {Small time global null controllability for a viscous Burgers’ equation despite the presence of a boundary layer}, J. Math. Pures Appl., {$\bf 102$} ($2014$), $364$–$384$.
%


\bibitem{Leray} 
{\sc	J. Leray}, {\it Sur le mouvement d’un liquide visqueux emplissant l’espace}, Acta Math., {$\bf63$} ($1934$), $193$-$248$. 

\bibitem{lions} 
{\sc J.-L. Lions and E. Magenes}, {\it Non-Homogeneous Boundary Value Problems and Applications}, Translated from the French by P. Kenneth. Die Grundlehren der Mathematischen Wissenschaften,  Springer-Verlag, New York–Heidelberg, {$\bf 2$} ($1972$).

\bibitem{simon} 
{\sc	J. Simon}, {\it Compact sets in the space $L^p(0,T;B)$}, Ann. Mat. Pura Apply., {$\bf 146$} ($1987$), $65$-$96$.

\end{thebibliography}
\end{document}